\documentclass[11pt]{amsart}

\title{Universal sets for projections}

\author{Jacob B. Fiedler}
\address{Department of Mathematics, University of Wisconsin, Madison, Wisconsin 53715}

\email{jbfiedler2@wisc.edu}
\thanks{The first author was supported in part by NSF DMS-2037851 and NSF DMS-2246906.}
\author{D. M. Stull}
\address{Department of Mathematics, University of Chicago, Chicago, IL 60637}
\email{dmstull@uchicago.edu}


\usepackage{amsmath}
\usepackage{amssymb}
\usepackage{mdframed}
\usepackage{amsthm}
\usepackage{mathtools}
\usepackage{enumitem}
\usepackage{placeins}
\usepackage{url}
\usepackage{hyperref}

\newtheorem{thm}{Theorem}

\newtheorem{lem}[thm]{Lemma}
\newtheorem{prop}[thm]{Proposition}
\newtheorem{cor}[thm]{Corollary}

\theoremstyle{remark}
\newtheorem*{remark}{Remark}

\DeclareMathOperator{\Dim}{Dim}
\DeclareMathOperator{\dimH}{dim_H}

\newcommand{\R}{\mathbb{R}}

\newcommand{\N}{\mathbb{N}}
\newcommand{\Q}{\mathbb{Q}}

\newcommand{\ve}{\varepsilon}

\newcommand{\ML}{\text{ML}}

\begin{document}

\begin{abstract}
    We investigate variants of Marstrand's projection theorem that hold for sets of directions and classes of sets in $\mathbb{R}^2$. We say that a set of directions $D \subseteq\mathcal{S}^1$ is \emph{universal} for a class of sets if, for every set $E$ in the class, there is a direction $e\in D$ such that the projection of $E$ in the direction $e$ has maximal Hausdorff dimension. We construct small universal sets for certain classes. Particular attention is paid to the role of regularity. We prove the existence of universal sets with arbitrarily small positive Hausdorff dimension for the class of weakly regular sets. We prove that there is a universal set of zero Hausdorff dimension for the class of AD-regular sets. 
\end{abstract}

\maketitle
\section{Introduction}
Marstrand's projection theorem is one of the most prominent results in geometric measure theory. It states that for analytic $E\subseteq\mathbb{R}^n$, almost every direction $e\in S^{n-1}$ satisfies
\begin{equation}\label{eq:Marstrand}
    \dim_H(p_e E) = \min \{1, \dim_H(E)\}
\end{equation}
where $p_e x = e\cdot x$ denotes the projection of $x$ onto $e$ and $p_e E$ is the image of $E$ under this projection map \cite{Marstrand54}. Projections cannot increase the dimension of a set, nor can the Hausdorff dimension of a subset of $\R$ exceed $1$. Thus, Marstrand's theorem says that an analytic set's projections are as large as possible in almost every direction.

Variants and generalizations of this theorem have been a fruitful research direction essentially since its original proof. Mattila \cite{Mattila75} extended Marstrand's result to projections onto higher dimensional linear subspaces. Lutz and the second author \cite{LutStu24} generalized the original theorem to weakly regular sets, i.e., sets with equal Hausdorff and packing dimension. Orponen \cite{Orponen20a} extended their result to higher dimensional orthogonal projections of weakly regular sets. Recently, the second author \cite{Stull22a} further generalized \eqref{eq:Marstrand} to hold for sets with ``optimal oracles'', a class that strictly contains weakly regular sets, analytic sets, and sets with metric outer measures.

An equivalent formulation of Marstrand's theorem is the following. If $E\subseteq \R^2$ is analytic and $\mathcal{D}\subseteq\mathcal{S}^1$ such that $\dim_H(p_e E) < \min\{\dim_H(E), 1\}$ for all $e\in \mathcal{D}$, then $\mathcal{D}$ must be of (Lebesgue) measure zero. We consider the converse of this formulation of Marstrand's theorem. Namely, if $\mathcal{D}\subseteq\mathcal{S}^1$, is there an analytic set $E\subseteq\R^2$ such that $\dim_H(p_e E) < \min\{\dim_H(E), 1\}$ for all $e\in \mathcal{D}$? More generally, let $\mathcal{C}$ be a class of subsets of $\R^2$. We say that a set of directions $\mathcal{D}\subseteq\mathcal{S}^1$ is \textit{universal for $\mathcal{C}$} if, for every $E\in\mathcal{C}$, there is some $e\in \mathcal{D}$ such that $\dim_H(p_e E) = \min\{\dim_H(E), 1\}$. Note that Marstrand's projection theorem immediately implies any positive measure set of directions is universal for the class of analytic sets. In this paper, we consider the existence of small (in a dimensional sense) universal sets for a variety of classes. 

The main results of the paper show that, for classes of regular subsets of $\R^2$, there are very small universal sets. In particular, we show that there is a set $\mathcal{D}\subseteq\mathcal{S}^1$ of lower box counting (and therefore Hausdorff) dimension zero which is universal for the class of AD-regular sets. In addition, we prove that, for every $\ve > 0$, there is a set $\mathcal{D}\subseteq\mathcal{S}^1$ of lower box counting dimension at most $\ve$ which is universal for the class of weakly regular sets. This is explained in more detail in Section \ref{ssec:introuniversalregular}. 

Universal sets for projections have been investigated in higher dimensions in the context of the ``restricted'' projection problem. Working in $\mathbb{R}^3$, F\"assler and Orponen conjectured that if $\gamma: I\to S^2$ is a suitable curve, then for any analytic set $E$, almost every $\theta\in I$ is such that $p_{\gamma(\theta)}(E)$ has maximal Hausdorff dimension. In other words, (positive measure subsets of) such curves are universal for the class of analytic sets in $\mathbb{R}^3$. A more detailed discussion of progress on this problem is contained in the next subsection.

We remark that the concept of ``universal sets" is not unique to projection theorems. The question of the existence of universal sets is well defined for any almost everywhere theorem. Indeed, this has been very successfully investigated in the context of Rademacher's theorem, on the differentiability properties of Lipschitz functions. Preiss \cite{Preiss90} showed that there is a Lebesgue null set $A\subseteq\R^2$ which is universal for Lipschitz functions $f:\R^2\rightarrow \R$, i.e., for any such $f$, there is an $x\in A$ at which $f$ is differentiable. This has been extended in \cite{PreSpe15,AlbCsoPre10} for a complete classification in all dimensions. The authors have also studied universal sets in the context of the pinned distance set problem, where regularity similarly plays an important role \cite{FieStu23, FieStu24}. Thus the existence of universal sets is an interesting phenomenon which is present in many areas of geometric measure theory, and analysis in general. 

\subsection{The restricted projection problem}

A key difference between $\mathbb{R}^3$ and $\mathbb{R}^2$ is that $\mathcal{S}^2$, the set of directions in $\mathbb{R}^3$, contains ``nice'' small subspaces i.e. curves. The restricted projection problem asks under what conditions a version of Marstrand's projection theorem holds on these subspaces. We clearly need some condition on the curves, as the following basic example shows. Let $\gamma(\theta)=\{\sin(\theta), \cos(\theta), 0\}$ and let $Z$ be the $z$-axis in $\mathbb{R}^3$. Then for every $\theta\in[0, 2\pi)$, $p_{\gamma(\theta)}Z$ is a singleton. F{\"a}ssler and Orponen \cite{FasOrp14} conjectured that the following non-degeneracy condition is sufficient for $C^2$ curves $\gamma: I\to \mathcal{S}^2$ to admit a version of Marstrand's projection theorem onto lines:
\begin{equation*}
    \text{span}\{\gamma(\theta), \gamma^\prime(\theta), \gamma^{\prime\prime}(\theta)\}=\mathbb{R}^3 \qquad \forall \theta\in I
\end{equation*}
Moreover, they conjectured that for such a $\gamma$, the space of its orthogonal complements also admits a version of Marstrand's projection theorem; note that these are projections onto \textit{planes}. We summarize the work on this problem that first implied the universality of certain sets. We begin by considering projections onto lines in $\mathbb{R}^3$:
\begin{itemize}
    \item Nondegenerate $C^2$ curves are universal for the class of analytic sets of dimension $\leq \frac{1}{2}$ \cite{FasOrp14, Kaufman68}
    \item Nondegenerate $C^2$ curves are universal for the class of self-similar sets without rotations \cite{FasOrp14, Hochman14}
    \item Nondegenerate circles are universal for the class of analytic sets \cite{KaeOrpVen17}
    \item Nondegenerate $C^2$ curves are universal for the class of analytic sets \cite{PraYangZahl, GanGuthMal2023}
    \item Certain AD-regular sets of dimension more than 1 are universal for the class of analytic sets \cite{Chen18}
\end{itemize}
For projections onto planes in $\mathbb{R}^3$, we have that:
\begin{itemize}
    \item The orthogonal complements of nondegenerate $C^2$ curves are universal for the class of analytic sets of dimension $\leq 1$ \cite{FasOrp14, Kaufman68}
    \item The orthogonal complements of nondegenerate $C^2$ curves are universal for the class of analytic sets \cite{GanGuoGuthHarMalWang2024}
    \item Certain AD-regular sets of dimension more than 1 are universal for the class of analytic sets \cite{Chen18}
\end{itemize}

The restricted projection problem can also be posed in higher dimensions; \cite{GanGuoWang2022} essentially establishes that nondegenerate curves give rise to universal sets for projections from $\mathbb{R}^n$ onto $k$-planes. A further discussion of some of the partial progress on this problem can be found in \cite{GanGuoWang2022}. 

\subsection{Exceptional set estimates and universal sets}\label{ssec:exceptionalsetestimates}
Returning to $\mathbb{R}^2$, because small subspaces of $\mathcal{S}^1$ are necessarily discontinuous fractals, much less is known than in the case of restricted projections. In every dimension, however, the existence of universal sets is related to exceptional set estimates, a central topic in geometric measure theory. Given a set $E\subseteq\R^2$ and $0 < s \leq \min\{\dim_H(E), 1\}$, define the exceptional set of $E$ by
\begin{center}
    $D_s(E) := \{ e\in\mathcal{S}^1\mid \dim_H(p_e E) < s\}$.
\end{center}
In 1968, Kaufman \cite{Kaufman68} refined Marstrand's projection theorem, by showing that, if $E\subseteq\R^2$ is analytic and $0\leq s\leq \min\{\dim_H(E), 1\}$, 
\begin{equation*}
    \dim_H\{D_s(E)\} \leq s.
\end{equation*}
The existence of sufficiently strong exceptional set estimates imply the existence of (small) universal sets.

We say a set of Hausdorff dimension $s$ that cannot be written as a union of sets of strictly smaller dimension has \textit{essential Hausdorff dimension} $s$. Applying Kaufman's theorem, it is easy to see that any set of directions with essential Hausdorff dimension $1$ is universal for the class of analytic sets. Briefly, given an analytic set $E$, let $d=\min\{\dim_H(E), 1\}$. Then, if $D(E) = \cup_{s<d} D_s(E)$ is the set of exceptional directions for $E$, we have
\begin{equation*}
    D(E) = \bigcup_{i\in\mathbb{N}} \{\theta\in S^1: \dim_H(p_\theta E)<d-\frac{1}{i}\}.
\end{equation*}
That is, $D(E)$ does not have essential Hausdorff dimension $1$ for any analytic set $E$. We may therefore conclude that any set $\mathcal{D}\subseteq\mathcal{S}^1$ of essential Hausdorff dimension $1$ is universal for the class of analytic sets. 

In this paper, we generalize Kaufman's exceptional estimate to a broader class of sets. As mentioned, in \cite{Stull22a} the second author introduced sets with ``optimal oracles'', a very general notion which we will define precisely in Section 2.2. Let 
\begin{equation*}
    \mathcal{C}_{OO} = \{E\subseteq\mathbb{R}^2: E\text{ has optimal oracles}\}
\end{equation*}
and
\begin{equation*}
    \mathcal{C}_{OO}(s) = \{E\subseteq\mathbb{R}^2: E\text{ has optimal oracles and } \dim_H(E)\leq s\}.
\end{equation*}
We prove the following
\begin{thm}
  Let $E\subseteq\R^2$ be a set with optimal oracles. Let $0 < s \leq \min\{\dim_H(E), 1\}$. Then,
    \begin{center}
        $\dim_H(\{e\in\mathcal{S}^1\mid \dim_H(p_e E) < s\}) \leq s$.
    \end{center}
\end{thm}

\noindent As above, this immediately implies the existence of small universal sets for $\mathcal{C}_{\text{OO}}$ and $\mathcal{C}_{\text{OO}}(s)$. 

Although exceptional sets and universal sets are clearly connected, there are important differences. The mere existence of universal sets of a certain size does \emph{not} imply the corresponding exceptional set estimates. On the other hand, the small universal sets we discuss in the remainder of this section (unlike the above) do not come from exceptional set estimates; rather they are specific sets of directions with certain properties, such as large size at well-chosen scales. 
\subsection{Universal sets for classes of regular sets}\label{ssec:introuniversalregular}
 The main theorems of this paper prove the existence of very small universal sets for classes of regular sets. 

Informally, we say that a set is regular if it appears to be the same size when viewed at different scales. There are several common notions of regularity. Ahlfors-David regularity is a rather strong notion. Less restrictive is weak regularity. We say that a set is weakly regular if its Hausdorff and packing dimension are equal. For ease of reference, denote
\begin{equation*}
\mathcal{C}_{AD} = \{E\subseteq\mathbb{R}^2: E \text{ is AD-regular }\}
\end{equation*}
and 
\begin{equation*}
\mathcal{C}_{WR} = \{E\subseteq\mathbb{R}^2: E \text{ is weakly regular }\}.
\end{equation*}

Our first main theorem is the following:
\begin{thm}
    There exists a set $\mathcal{D}_0 \subseteq S^1$ of lower box (and hence Hausdorff) dimension zero that is universal for $\mathcal{C}_{AD}$.
\end{thm}
We note that there is precedent for regularity assumptions resulting in improved projection theorems. In \cite{Orponen16}, Orponen showed that the class of Hausdorff dimension $1$ AD-regular sets satisfies a certain universality property for \textit{box} dimension. Recently, Orponen \cite{Orponen24} proved bounds on exceptional sets for AD-regular sets in the context of box dimension. For even stronger notions of regularity, better estimates are known. Hochman \cite{Hochman14} proved that the exceptional set for self-similar sets is countable, and Peres and Shmerkin \cite{PerShm2007} showed that if the self-similar set is generated by irrational rotations, the exceptional set is \emph{empty}  (hence, any singleton is universal for this class). On the other hand, Orponen and Ren \cite{OrpRen24} investigated the extent to which these exceptional set estimates fail for weaker notions than AD-regularity.

Although \cite{OrpRen24} in general rules out strong exceptional set estimates when we slightly weaken the regularity assumption, there is still hope that small universal sets for such classes exist. Indeed, our second main theorem shows that this is true. 
\begin{thm}
    For every $\ve>0$, there exists a set $\mathcal{D}_\ve \subseteq S^1$ of lower box dimension $\ve$ that is universal for $\mathcal{C}_{WR}$.
\end{thm}

Moreover, if we restrict the dimension of the sets in each class, we are able to also guarantee that the \emph{packing} dimension of each universal set is not too large.
\begin{prop}\label{prop:packingRegular}
    Let $s\in(0, 1)$,
    \begin{enumerate}
        \item There exists some $\mathcal{D} \subseteq S^1$ of lower box dimension 0 and packing dimension $s$ that is universal for $\mathcal{C}_{AD}\cap \mathcal{C}_s$
        \item For every $\ve>0$ there exists some $\mathcal{D}_{\ve} \subseteq S^1$ of lower box dimension $\ve$ and packing dimension $s$ that is universal for $\mathcal{C}_{WR}\cap \mathcal{C}_s$ 
    \end{enumerate}
\end{prop}

We also note some very natural examples of universal sets for regular classes. Define

\begin{equation*}
    H_s =\{e\in \mathcal{S}^1: \dim(e) = s\}
\end{equation*}

\noindent That is, $H_s$ is simply the set of directions of effective Hausdorff dimension $s$ (a notion we define in the next section). It follows from the point-to-set principle (theorem \ref{thm:pointToSet}) that $\dimH(H_s)=s$. From our definitions of $\mathcal{D}_s$ in section \ref{sec:smallsets}, it will be clear that for any $s\geq 0$, $\mathcal{D}_s\subseteq H_s$. Hence, $H_0$ is universal for the class of AD regular sets, and for any $s>0$, $H_s$ is universal for the class of weakly regular sets.

\subsection{Bourgain universal sets}
In this paper, we consider a slightly different formulation of the concept of universal sets, inspired by Bourgain's exceptional set estimate. In \cite{Bourgain10}, Bourgain proved a remarkable theorem, an exceptional set estimate showing that, for any analytic set $E\subseteq\R^2$,
\begin{equation}
    \dim_H(\{e \in \mathcal{S}^1\mid \dim_H(p_e E) \leq \dim_H(E)/2\} = 0.
\end{equation}
Call a set $\mathcal{D} \subseteq \mathcal{S}^1$ \textit{Bourgain universal for a class $\mathcal{C}$} if, for every set $E\in\mathcal{C}$, there is a direction $e \in \mathcal{D}$ such that
\begin{equation}
    \dim_H(p_e E) \geq \frac{\dim_H(E)}{2}.
\end{equation}

Bourgain's exceptional set estimate implies every set with \emph{nonzero} Hausdorff dimension is Bourgain universal for the analytic sets. However, it does not immediately imply that zero dimensional Bourgain universal sets exist. We establish the existence of dimension zero universal sets with the following theorem.  
\begin{thm}
    There exists some $\mathcal{D}\subseteq S^1$ of lower box dimension $0$ that is Bourgain universal for the class of analytic sets
\end{thm}
\noindent We also prove a related theorem for the broader class of sets with optimal oracles. However, by weakening the assumption to sets with optimal oracles, we only get arbitrarily small (Bourgain) universal sets. More precisely
\begin{prop}\label{prop:OOBourgain}
    For every $\ve>0$, there exists some $\mathcal{D}\subseteq S^1$ of Hausdorff dimension $\ve$ that is Bourgain universal for the class of sets with optimal oracles. 
\end{prop}

Recently Ren and Wang proved a sharp exceptional set estimate containing the results of Kaufmann and Bourgain and verifying a conjecture of Oberlin. Namely, they showed that if $E\subseteq\mathbb{R}^2$ is Borel and $u$ is such that $0\leq u\leq \min\{1, \dim_H(E)\}$, then
\begin{equation*}
    \dim_H\{\theta\in S^1: \dim_H(\pi_\theta E)<u\}\leq \max\{2u - \dim_H(E), 0\}.
\end{equation*}

For the purpose of completeness, we note that this result implies the existence of small ``generalized'' universal sets in the same manner as the Kauffman and Bourgain estimates respectively imply the existence of small universal and Bourgain universal sets.

\section{Preliminaries}\label{sec:effprojresults}
\subsection{Geometric measure theory}
We briefly recall the classical notions of Hausdorff and packing dimension. Given $E\subseteq\mathbb{R}^n$, call a collection of open balls $\{B_i\}_{i\in\mathbb{N}}$ a $\delta$-cover for $E$ if $E\subseteq\bigcup_{i\in\mathbb{N}}B_i$ and $\text{diam}(B_i)\leq \delta$ for all $i$. Likewise, call $\{B_i\}_{i\in\mathbb{N}}$ a $\delta$-packing for $E$ if each ball has its center in $E$, the balls are pairwise disjoint, and $\text{diam}(B_i)\leq \delta$ for all $i$. The $s$-dimensional Hausdorff (outer) measure of $E$ is 
\begin{equation*}
    \mathcal{H}^s(E) = \lim_{\delta\to 0^+} \inf_{\substack{\delta-\text{covers}\\ \text{ of }E}}\limits\{\sum\limits_{i=1}^\infty\text{diam}(B_i)^s\}
\end{equation*}
while the $s$-dimensional packing pre-measure of $E$ is 
\begin{equation*}
    \bar{\mathcal{P}}^s(E) = \lim_{\delta\to 0^+} \sup_{\substack{\delta-\text{packings}\\ \text{ of }E}}\limits\{\sum\limits_{i=1}^\infty\text{diam}(B_i)^s\}.
\end{equation*}
We can use the above pre-measure to define a genuine (outer) measure as follows:
\begin{equation*}
    \mathcal{P}^s(E) = \inf\{\sum\limits_{i=1}^\infty \bar{\mathcal{P}}^s(E_i): E\subseteq\bigcup\limits_{i=1}^\infty E_i\}.
\end{equation*}
Then, the packing and Hausdorff dimensions are
\begin{equation*}
    \dim_H(E) = \inf\{s: \mathcal{H}^s(E) = 0\} \qquad \text{and} \qquad 
 \dim_P(E) = \inf\{s: \mathcal{P}^s(E) = 0\}. 
\end{equation*}

\subsection{Effective dimension and the point-to-set principle}

We begin by reviewing the definitions of effective dimension. More a more detailed discussion, we refer to \cite{LutLut18}. The Kolmogorov complexity of a string $\sigma\in\{0, 1\}^*$ is the length of the shortest string that is a description for $\sigma$ with no additional information, in the following sense. Let $\vert \pi \vert$ denote the length of a binary string $\pi$ and let $\lambda$ denote the empty string. Fix a universal (prefix-free) Turing machine $U$. The Kolmogorov complexity of $\sigma$ is 
\begin{equation*}
    K(\sigma) = \min_{\pi \in\{0, 1\}^*}\{ \vert \pi \vert: U(\pi, \lambda)=\sigma \}.
\end{equation*}
More generally, we can consider the complexity of $\sigma$ if $U$ is also given a string $\tau$. The Kolmogorov complexity of $\sigma$ conditioned on $\tau$ is
\begin{equation*}
    K(\sigma\vert \tau) = \min_{\pi \in\{0, 1\}^*}\{ \vert \pi \vert: U(\pi, \tau)=\sigma \}.
\end{equation*}
Using a fixed encoding of rational vectors as strings, we can naturally define $K(p)$ for $p\in\mathbb{Q}^n$. This, in turn, allows us to define the Kolmogorov complexity of points in $\mathbb{R}^n$ at a certain precision. Let $x\in\R^n$, and $r\in\N$. The Kolmogorov complexity of $x$ at precision $r$ is
\begin{equation*}
K_r(x) = \min\{K(p): p\in\mathbb{Q}^n\cap B_{2^{-r}}(x)\},
\end{equation*}
Thus, the complexity of $x$ at precision $r$ is the complexity of the simplest rational vector that is a precision $r$ approximation of $x$. Similarly, we can define the conditional Kolmogorov complexity of points. Let $x\in\R^n$, $y\in \R^m$ and $r,s\in\N$. The Kolmogorov complexity of $x$ at precision $r$ conditioned on $y$ at precision $s$ is
\begin{equation*}
    K_{r, s}(x\vert y) = \max\{\min\{K(p\vert q): p\in \mathbb{Q}^n\cap B_{2^-r}(x)\}: q\in \mathbb{Q}^m\cap B_{2^{-s}}(y)\}.
\end{equation*}
To save space, we often write $K_r(x\vert y):=K_{r, r}(x\vert y)$ and $K_{r, s}(x):=K_{r, s}(x\vert x)$. We now introduce the several properties of Kolmogorov complexity that we will frequently use. \\

\noindent \textbf{Symmetry of information:}
For finite binary strings, the symmetry of information states that $K(\sigma, \tau) = K(\sigma) + K(\tau\vert \sigma, K(\sigma)) + O(1)$. In \cite{LutStu20}, Lutz and the second author proved that a version of this holds in Euclidean space, namely
\begin{equation*}
    \vert K_{r,s}(x, y) -K_{r, s}(x\vert y) -K_s(y)\vert \leq O(\log r)+ O(\log s) + O(\log \log \vert y\vert). 
\end{equation*}

An important application of the symmetry of information is the following. If we let $y=x$, this enables us to ``partition'' the interval $[1, r]$ in the following sense. Let a sequence of precisions $1=r_0, r_1, r_2, ..., r_m=r$ be given. Then
\begin{align*}
    K_r(x) &= \sum\limits_{i=1}^m \left(K_{r_i}(x) - K_{r_{i-1}}(x)\right) + K_1(x)\\
    &\approx \sum_{i=1}^m K_{r_i, r_{i-1}}(x)
\end{align*}
where the second equality holds with an $O(m\log r)$ error term.\\

\noindent \textbf{Bounds on complexity growth:}
It is immediate from the definition of Kolmogorov complexity at a precision that for any $r, s,$ and $x$,
\begin{equation*}
    K_{r+s}(x) \geq K_r(x).
\end{equation*}
We also have the following upper bound on this quantity due to Case and Lutz \cite{CasLut15}
\begin{equation}\label{eq:caseLutz}
K_{r+s}(x) - K_{r}(x) \leq n s + O_{n}(\log r)+ O_{n}(\log s).
\end{equation}
We say a point $x$ is (Martin-Lof) random if its complexity is close to this upper bound at every precision, in particular if there exists some $c$ such that $K_r(x)\geq nr -c$ for all $r$. We note that set of ML random points has full Lebesgue measure in $\mathbb{R}^n$. \\

\noindent \textbf{Oracles:}
By giving the universal Turing machine $U$ access to an oracle $A\subseteq\mathbb{N}$, the above definitions can be ``relativized'', which we indicate with a superscript. This oracle contains information which the machine $U$ can use in its computations. All of the above facts are true relative to an arbitrary oracle. Adding access to a oracle can never make a computation appreciably harder, since the machine is able to, at worst, ignore the information in the oracle. In particular
\begin{equation*}
    K_r^{A, B}(x)\leq K^A_r(x) + O(\log r),
\end{equation*}
where the superscript $A, B$ indicates the Turing machine has access to the information in $A$ \emph{and} $B$. Using a standard encoding, we can consider points in $\mathbb{R}^n$ as oracles. Equipped with such an oracle $y$, $U$ can be thought of as having unrestricted access to its binary expansion, whereas with conditional complexity, the access is restricted to a certain precision. Hence, 
\begin{equation*}
    K_r^{A, y}(x)\leq K_{r, s}^{A}(x\mid y) + O(\log r)\leq K_{r}^{A}(x) + O(\log r).
\end{equation*}

We will often need to reduce the complexity of some point $x$ at a specified precision. However we want to do so in a way that does not lower the complexity of other objects, at least not any more than information about $x$ does. The following lemma of Lutz and the second author \cite{LutStu24} introduces an oracle $D$ which accomplishes this goal. 
\begin{lem} \label{lem:finite-oracle}
    Let $A \subseteq \{0,1\}^*$, $r \in \N$, $x \in \R^n$, and $\eta \in \Q^+$. There is an oracle $D = D(A, n, r, x, \eta)$ satisfying the following:
    \begin{enumerate}[label={\normalfont \textbf{(\arabic*)}}, itemsep=1.5pt, topsep=-3pt]
        \item For every natural number $t \leq r$, 
        \begin{equation*}
            K_t^{A,D}(x) = \min \, \{ \eta r, K_t^A(x) \} + O(\log r).
        \end{equation*}
        \item For every $m,t \in \N$ and $y \in \R^m$, 
        \begin{equation*}
            K_{t,r}^{A,D}(y \mid x) = K_{t,r}^A(y \mid x) + O(\log r) \quad \text{and} \quad K_t^{A,D,x}(y) = K_t^{A,x}(y) + O(\log r).
        \end{equation*}
        \item If $B \subseteq \{0,1\}^*$ satisfies $K_r^{A,B}(x) \geq K_r^A(x) - O(\log r)$
        \begin{equation*}
            K_r^{A,B,D}(x) \geq K_r^{A,D}(x) - O(\log r).
        \end{equation*}
        \item For every $m,t \in \N$, $u \in \R^n$, and $w \in \R^m$,
        \begin{equation*}
            K_{r,t}^A(u \mid w)\leq K_r^{A,D}(u \mid w) + K_{r,t}^A(x) - \eta r + O(\log r).
        \end{equation*}
    \end{enumerate}
\end{lem}

\noindent \textbf{The point-to-set principle:}
In this paper, the main application of Kolmogorov complexity is to define a notion of asymptotic complexity of points in Euclidean space. In particular, the effective Hausdorff dimension of $x\in\mathbb{R}^n$ relative to an oracle $A$ is
\begin{equation*}
    \dim^A(x) = \liminf \frac{K^A_r(x)}{r}
\end{equation*}
and the effective packing dimension of $x\in\mathbb{R}^n$ relative to $A$ is
\begin{equation*}
    \Dim^A(x) = \limsup \frac{K^A_r(x)}{r}.
\end{equation*}

The effective and classical notions of dimension are closely connected through the point-to-set principle of Lutz and Lutz \cite{LutLut18}.
\begin{thm}\label{thm:pointToSet}
    For every $E\subseteq\mathbb{R}^n$,
    \begin{equation}\label{eq:HPtS}
        \dim_H(E) = \min_{A\subseteq\mathbb{N}}\sup_{x\in E}\dim^A(x)
    \end{equation}
    and
    \begin{equation}\label{eq:PPtS}
        \dim_P(E) = \min_{A\subseteq\mathbb{N}}\sup_{x\in E}\Dim^A(x).
    \end{equation}
\end{thm}
 We call an oracle $A$ a Hausdorff oracle for $E$ if it achieves the minimum in \eqref{eq:HPtS} and a packing oracle if it achieves the minimum in \eqref{eq:PPtS}. Thus, every set has a Hausdorff and packing oracle. We note that the join of a Hausdorff (packing) oracle for $E$  with an arbitrary oracle is still a Hausdorff (packing) oracle for $E$. \\

\noindent \textbf{Optimal oracles:} By the point-to-set principle, a Hausdorff oracle $A$ for $E$ has the property that, for any additional oracle $B$ and $\ve>0$, there is some $x\in E$ such that $\dim^{A, B}(x)\geq \dim^A(x)-\ve$. In other words, there is no oracle $B$ that is significantly ``more helpful'' for computing every point in $E$ than $A$, in the sense of effective Hasudorff dimension. In \cite{Stull22a}, the second author introduced a stronger notion capturing when additional oracles $B$ are similarly unhelpful \emph{at every precision} (not just after taking a limit inferior). To be more precise, we say an oracle $A\subseteq\mathbb{N}$ is Hausdorff optimal for $E\subseteq \R^n$ if
\begin{enumerate}
    \item $A$ is a Hausdorff oracle for $E$.
    \item For every $B\subseteq\mathbb{N}$ and $\varepsilon>0$, there is some $x\in E$ such that $\dim^{A, B}(x)>\dim_H(E)-\varepsilon$ and for all sufficiently large $r$, 
    \begin{equation*}
        K^{A, B}_r(x)\geq K^{A}_r(x) -\varepsilon r
    \end{equation*}
\end{enumerate}

We note that the join of an optimal oracle and an arbitrary oracle is optimal. Most ``reasonable'' sets have optimal oracles, including the analytic (and hence Borel) sets, the weakly regular sets, and sets with metric outer measures.

\section{Effective projections}\label{sec:effproj}
In this section, we turn our attention specifically towards the complexity of projections. We begin by stating two theorems we will employ directly. 

Essentially due to Lutz and the second author \cite{LutStu22}, this result gives sufficient conditions for strong lower bounds on the complexity of projected points. 
\begin{thm}\label{thm:LutStu22}
Let $A,B\subseteq \N$, $x \in \R^2$, $e \in \mathcal{S}^1$, $s \in (0, 1]$, $\ve > 0$ and $r\in \N$. Assume the following are satisfied.
\begin{enumerate}
\item For every $t \leq r$, $K^A_{t}(e) \geq st - \ve r$.
\item $K^{A, B, e}_r(x) \geq K^A_r(x) - \ve r$.
\end{enumerate}
Then,
\[K^{A,B, e}_r(p_e x ) \geq \min\{\dim^A(x), s\}r - 10\ve^{1/2} r - O(\log r)\,.\]
\end{thm}

\noindent In \cite{Stull22c}, the second author proved a related result.
\begin{thm}\label{thm:conditionalProjectionBounds}
 Let $x\in \mathbb{R}^2, e\in S^1, \ve> 0, C \geq 1, A\subseteq\mathbb{N}$, and $t\leq r\in \mathbb{N}$. Suppose that $r$ is sufficiently large and that the following hold
\begin{enumerate}
    \item $t\geq \frac{r}{C}$,
    \item $K_{s,r}^{A}(e\mid x)\geq s - \ve r$ for all $s\leq t$
\end{enumerate}
Then, 
\begin{equation*}
    K_r^A(x\mid p_ex, e)\leq \max\{K^A_r(x) - r, \frac{K^A_r(x) - t}{2}, 0\} + 10C\ve r,
\end{equation*}
\end{thm}

This theorem can be viewed as a generalization of the preceding theorem. Namely, in Theorem \ref{thm:LutStu22}, the direction $e$ is assumed to have high complexity at \emph{all} precisions, whereas in this case it only has high complexity up to some precision $t\leq r$. As a result of this restriction, the argument in \cite{Stull22c} involves ``partitioning'' to isolate desirable complexity properties on certain types of intervals of precisions, which we now define.

Let $A\subseteq\N$, $x\in\R^2$, and $\sigma\in (0,2]$. Let $a \leq b$. We say that $[a,b]$ is \textit{$(\sigma,c)$-teal} if 
\begin{equation}
    K^A_{b, s}(x\mid x) \leq \sigma(b-s) + c \log b,
\end{equation}
for all $a\leq s \leq b$. We say that $[a,b]$ is \textit{$(\sigma, c)$-yellow} if
\begin{equation}
    K^A_{s,a}(x\mid x) \geq \sigma(s-a) - c \log b,
\end{equation}
for all $a\leq s \leq b$.

In this paper, we will require more refined versions of Theorem \ref{thm:conditionalProjectionBounds}. In particular, we will need  refined bounds on the complexity of computing $x$ given its projection $p_e x$ on yellow and teal intervals. 

The following three lemmas strengthen the results of \cite{Stull22c} which are used in the proof of Theorem \ref{thm:conditionalProjectionBounds}. These lemmas will be essential in the proofs of Theorems \ref{thm:CountablyUniversalWeaklyRegular}, \ref{thm:universalSetsADRegular}, and \ref{thm:BourgainUniversal}. 

\begin{lem}\label{lem:projectionYellowTeal}
Let $x\in\R^2, e\in\mathcal{S}^1$, $c\in\N$, $\sigma\in\mathbb{Q}\cap(0, 1]$, $A\subseteq\mathbb{N}$ and $a<b\in\R_+$. Suppose that $b$ is sufficiently large (depending on $e$, $x$, and $\sigma$) and $K^A_{s, b}(e\mid x) \geq \sigma s - c\log b$, for all $s\leq b-a$. Then the following hold.
\begin{enumerate}
\item If $[a,b]$ is $(\sigma, c)$-yellow, 
\begin{center}
$K^A_{b,b,b,a}(x\mid p_e x, e,x) \leq K^A_{b,a}(x\mid x) - \sigma (b-a) + O_{c}(\log b)^2 $.
\end{center}
\item If $[a,b]$ is $(\sigma, c)$-teal, 
\begin{center}
$K^A_{b,b,b,a}(x\mid p_e x, e,x) \leq O_{c}(\log b)^2$.
\end{center}
\end{enumerate}
\end{lem}

We will also need the following, almost identical, lemma which only requires a straightforward modification of the proof of Lemma \ref{lem:projectionYellowTeal}, namely the introduction of a larger error term in the first portion of the proof of Lemma \ref{lem:pseudoYellow} (and \ref{lem:pseudoTeal}).

\begin{lem}\label{lem:alternateProjectionYellowTeal}
Let $x\in\R^2, e\in\mathcal{S}^1$, $c\in\N$, $\sigma\in\mathbb{Q}\cap(0, 1]$, $A\subseteq\mathbb{N}$ and $a<b\in\R_+$. Suppose that $b$ is sufficiently large (depending on $e$, $x$, and $\sigma$) and $K^A_{s, b}(e\mid x) \geq \sigma s - c\sqrt{b}$, for all $s\leq b-a$. Then the following hold.
\begin{enumerate}
\item If $[a,b]$ is $(\sigma, c)$-yellow, 
\begin{center}
$K^A_{b,b,b,a}(x\mid p_e x, e,x) \leq K^A_{b,a}(x\mid x) - \sigma (b-a) + O_{c}(\sqrt{b}) $.
\end{center}
\item If $[a,b]$ is $(\sigma, c)$-teal, 
\begin{center}
$K^A_{b,b,b,a}(x\mid p_e x, e,x) \leq O_{c}(\sqrt{b})$.
\end{center}
\end{enumerate}
\end{lem}
Lemmas \ref{lem:projectionYellowTeal} and \ref{lem:alternateProjectionYellowTeal} will complement each other in the following way. The first is useful when we can guarantee that the direction has essentially maximal complexity at some small precision (less than $r$), and we need to perform a partitioning argument with very little loss to obtain good bounds up to $r$. The second is useful when the direction has complexity that is only close to maximal, but \emph{at} the precision $r$ we are interested in. In this case, the larger error term is not harmful, since there is no partition, hence no accumulation of error. 

The above lemmas are useful when we need to be quite careful with our error terms, in particular when we want them to be much less than linear. However, the framework we use in our proof allows us to obtain bounds when we significantly loosen the requirements on the intervals of the complexity function for $x$. Let $a \leq b$. We say that $[a,b]$ is \textit{$(\sigma,\ve)$-almost teal} if 
\begin{equation}
    K^A_{b, s}(x\mid x) \leq \sigma(b-s) + \ve b,
\end{equation}
for all $a\leq s \leq b$. We say that $[a,b]$ is \textit{$(\sigma, \ve)$-almost yellow} if
\begin{equation}
    K^A_{s,a}(x\mid x) \geq \sigma(s-a) - \ve b,
\end{equation}
for all $a\leq s \leq b$. Note that we allow $\ve r$ error terms instead of the $\log r$ error terms for yellow and teal intervals. 

\begin{lem}\label{lem:almostProjectionYellowTeal}
Let $x\in\R^2, e\in\mathcal{S}^1$, $c\in\N$, $\sigma\in\mathbb{Q}\cap(0, 1]$, $A\subseteq\mathbb{N}$ and $a<b\in\R_+$. Suppose that $b$ is sufficiently large (depending on $e$, $x$, and $\sigma$) and $K^A_{s, b}(e\mid x) \geq \sigma s - c \log b$, for all $s\leq b-a$. Then the following hold:
\begin{enumerate}
\item If $[a,b]$ is $(\sigma, \ve)$-almost yellow, 
\begin{center}
$K^A_{b,b,b,a}(x\mid p_e x, e,x) \leq K^A_{b,a}(x\mid x) - \sigma (b-a) + 4 \ve b $.
\end{center}
\item If $[a,b]$ is $(\sigma, \ve)$-almost teal, 
\begin{center}
$K^A_{b,b,b,a}(x\mid p_e x, e,x) \leq 4\ve b$.
\end{center}
\end{enumerate}
\end{lem}

\subsection{Proof of main lemmas}
In this section, we will prove Lemma \ref{lem:projectionYellowTeal}. We again note that the proofs for Lemma \ref{lem:alternateProjectionYellowTeal} and Lemma \ref{lem:almostProjectionYellowTeal} are simple modifications of the proof, and are omitted.  

The proof of Lemma \ref{lem:projectionYellowTeal} proceeds in two main steps. First, we need an ``enumeration'' lemma, which gives technical conditions under which we can find a strong upper bound for the complexity growth of $x$ on $[a, b]$ given $e$ up to precision $b$.

\begin{lem}
    Let $x\in\mathbb{R}^2$, $B\subseteq\mathbb{N}$, $e\in S^1$, $a, b\in \mathbb{N}, \sigma, c\in \mathbb{R}_{+}$, and $ \eta\in \mathbb{Q}^+$ be such that $b\geq a$ is sufficiently large and the following conditions hold:
    \begin{enumerate}
        \item $K^B_b(x)\leq \eta b +  c_1 \log b$
        \item For every $w\in B_{2^{-a}}(x)$ such that $p_e w = p_e x$, 
        \begin{equation*}
        K^B_b(w)\geq \eta b + \min \{c_2(\log b)^2, \sigma(b-s)- c_1 \log b\}  
        \end{equation*}
        whenever $s=-\log \vert x - w\vert \in (a, b]$
    \end{enumerate}
    Then for every $A\subseteq\mathbb{N}$
    \begin{equation*}
    K_{b, b, a}^{A, B}(x\mid e, x)\leq K_{b, b, a}^{A, B}(p_e x\mid e, x) + \left(\frac{4(c_1+1)}{\sigma} + C_{x, e}\right)\log b + K^B(\eta, c). 
    \end{equation*}
\end{lem}

As compared to other enumeration lemmas, this one involves $O(\log b)$ instead of $o(b)$ error terms, though its proof is very similar, in that we construct a Turing machine $M$ with conditional access to $x$ up to precision $a$ and the direction $e$ up to precision $b$. After receiving $p_e x$ at precision $b$ as an input, $M$ attempts to compute $x$ up to precision $b$ by testing points until it finds a point that:
\begin{itemize}
\item Agrees with $x$ up to precision $a$, 
\item Has nearly the same projection in the direction $e$ as $x$, and 
\item Is no more complex, at precision $b$, than $x$.
\end{itemize}
However, the machine may output a ``false positive'', that is, some approximation of $w\neq x$. The second condition of the lemma will guarantee that even the false positives are relatively close to $x$.

\begin{proof}
If $b\leq a+2$, then $K_{b, b, a}^{A, B}(x\mid e, x)\leq O(\log b)$ by \eqref{eq:caseLutz} and we are done. Assuming otherwise, define an oracle Turing machine $M$ that does the following given oracle $A, B$ and inputs $\sigma = \sigma_1\sigma_2\sigma_3\sigma_4$ and $(\theta, q_1)\in [0, 2\pi]\cap\mathbb{Q}\times \mathbb{Q}^2$ such that $U^B(\sigma_2)=s_1$, $U^B(\sigma_3)=s_2$, and $U^B(\sigma_4)=(\zeta, \iota)$. 

First, $M$ computes $U^{A, B}(\sigma_1, (\theta, q_1))=p$. Then, $M$ calculates $U^{B}(\pi)=q_2$ for all programs $\pi$ of length less than or equal to $\iota s_2 + \zeta\lceil\log s_2\rceil$ in parallel. For each $q_2$, $M$ checks whether the following are satisfied. 
\begin{enumerate}
    \item[\textup{(C1)}] $\vert q_2-q_1\vert \leq 2^{-(s_1 +1)}$
    \item[\textup{(C2)}] $\vert p_\theta q_2 - p\vert \leq 2^{-(s_2-1)}$
\end{enumerate}
\noindent $M$ outputs the first $q_2$ satisfying both of these conditions that it finds. 

In what follows, we will need to work with $x, e, $ and $p_ex$ up to certain precisions. To that end, let $\sigma_1$ testify to $K^{A, B}_{b+2, b+2+\lceil\log \vert x\vert \rceil, a+2}(p_e x\vert e, x)$, $\sigma_2$ testify to $K^B(a)$, $\sigma_3$ testify to $K^B(b)$, and $\sigma_4$ testify to $K^B(c_1, \eta)$. Let $\theta$ be a rational such that $\vert\theta - e\vert<2^{-(b+2+\lceil\log \vert x\vert \rceil)}$ and let $q_1$ be a rational such that $\vert q_1 - x\vert<2^{-(a+2)}$. First, we show that $M$ halts on this input. By the first condition of the lemma, the string $\pi_x$ that testifies to $K^B_b(x)$ has length no more than $\eta b + c_1\log (b)$, hence $M$ either checks the output of this string or halts before doing so. If $M$ checks this string, we have
\begin{equation*}
    \vert q_2-q_1 \vert \leq \vert x - q_2 \vert + \vert q_1-x \vert \leq 2^{-b}+ 2^{-(a+2)} \leq 2^{-(a+1)} 
\end{equation*}
So (C1) is satisfied. Similarly, 
\begin{align*}
\vert p_\theta q_2 - p\vert &\leq \vert p_e x - p_\theta q_2\vert + \vert p - p_e x \vert\\
&\leq \vert p_e x - p_\theta x\vert + \vert p_\theta x - p_\theta q_2\vert + \vert p - p_e x \vert\\
&\leq \vert p_e x - p_\theta x\vert + \vert x - q_2\vert + \vert p - p_e x \vert\\
&\leq  2^{-(b+1)} + 2^{-b} + 2^{-(b+2)}\\
&\leq 2^{-(b-1)}
\end{align*}
so (C2) is satisfied and $M$ halts on input $\sigma$ and $(\theta$, $q_1)$. (C1) and (C2), along with observation 3.2 in \cite{LutStu24} gives us the existence of some $w\in B_{2^{-a}}(x)$ such that $p_ew=p_ex$ and $\vert w - q_2\vert \leq 2^{-(b-1)}$. The proximity of $q_2$ to $w$, along with several applications of \eqref{eq:caseLutz}, ensures that
\begin{align*}
    K^{B}_{b, b, a}(w\mid e, x)&\leq K^{B}_{b-1, b+2+\log \vert x\vert, a+2}(w\mid e, x) + O_{\vert x \vert}(\log b)\\
    &\leq \vert \sigma\vert + c_M+ O_{\vert x \vert}(\log b)\\
    &\leq K^{A, B}_{b, b, a}(p_e x\mid e, x) + K^B(b) + K^B(a) + K^B(\eta, c) + O_{\vert x \vert}(\log b) \\
    &\leq K^{A, B}_{b, b, a}(p_e x\mid e, x) + K^B(\eta, c) + O_{\vert x \vert}(\log b).
\end{align*}
By symmetry of information, the definition of $s$, and another application of \eqref{eq:caseLutz}, 
\begin{align*}
    K^{A, B}_{b, b, a}(x\mid e, x)&\leq K^{A, B}_{b}(x\mid w) + K^{A, B}_{b, b, a}(w\mid e, x)+ O_{\vert x\vert} (\log b)\\
    &\leq K^{B}_{b}(x\mid w) + K^{B}_{b, b, a}(w\mid e, x) + O_{\vert x\vert} (\log b)\\
    &\leq K^{B}_{b}(x\mid w) +  K^{A, B}_{b, b, a}(p_e x\mid e, x) + K^B(\eta, c)+ O_{\vert x\vert} (\log b)\\
    &\leq K^{B}_{b, s}(x) +  K^{A, B}_{b, b, a}(p_e x\mid e, x) + K^B(\eta, c)+ O_{\vert x\vert} (\log b)\\
    &\leq 2(b-s) +  K^{A, B}_{b, b, a}(p_e x\mid e, x) + K^B(\eta, c)+ O_{\vert x\vert} (\log b)
\end{align*}
So it remains to bound $(b-s)$, which we can do using condition 2 of the lemma. Because $M$ only tests strings of length no more than $\eta b + c_1\log b$, $K_{b-1}^B(w)\leq \eta b + c_1\log b$. Hence, for sufficiently large $b$
\begin{equation*}
K_{b}^B(w)\leq \eta b + (c_1+2)\log b< \eta b + c_2(\log b)^2,
\end{equation*}
 which contradicts condition 2 of the lemma unless $c_2(\log b)^2> \sigma(b-s) - c_1 \log b$. So we may assume $K_{b}^B(w)\geq \eta b + \sigma (b-s) - c_1\log b$. Comparing the upper and lower bounds on $K_b^B(w)$ gives
\begin{equation*}
    b - s\leq \frac{2c_2+2}{\sigma}\log b
\end{equation*}
which completes the proof. 

\end{proof}

Next, we state a geometric lemma of Lutz and the second author, which will enable us to show that the second condition of the enumeration lemmas is met on yellow and teal intervals. 

\begin{lem}[\cite{LutStu20}]\label{lem:intersectionLemmaProjections}
Let $A\subseteq\N$, $x \in \R^2$, $e \in S^{1}$, and $b \in \N$. Let $w \in \R^2$ such that  $p_e x = p_e w$ up to precision $b$. Then 
\begin{equation*}
    K^A_b(w) \geq K^A_b(x) + K^A_{b-a,b}(e\mid x) + O(\log b)\,,
\end{equation*}
where $a := -\log \vert x-w\vert$.
\end{lem}

\noindent Now, we are able to prove a preliminary lemma. 

\begin{lem}\label{lem:pseudoYellow}
    Let $x\in\mathbb{R}^2$, $e\in S^1$, $\sigma\in \mathbb{Q}\cap(0, 1]$, $c\in\mathbb{N}$, $A\subseteq\mathbb{N}$ and $a<b\in \mathbb{N}$. Assume $b$ is sufficiently large depending on $e$, $x$, and $\sigma$, and the following conditions are satisfied:
\begin{enumerate}
    \item $K_{s, b}^{A}(e\mid x)\geq \sigma s - c \log b$ for all $s\leq b-a$
    \item $K^A_{s, a}(x)\geq \sigma (s - a) - c \log b $ for all $a\leq s\leq b$
\end{enumerate}
Then
\begin{equation*}
K^A_{b,b, b, a}(x\mid p_e x, e, x)\leq K^A_{b, a}(x) - \sigma(b - a) + \frac{4c}{\sigma}\log b +3 (\log b)^2.
\end{equation*}
\end{lem}
\begin{proof}
We may assume $c\leq \frac{b}{\log b}$; otherwise the inequality is immediate. Let $\eta$ be such that $\eta b = K^A_a(x) + \sigma (b - a) - 2 (\log b)^2$ and let $D$ be the corresponding oracle of Lemma \ref{lem:finite-oracle}. First, note that there is some $C_1$ such that 
\begin{equation*}
    K^{A, D}_b(x) \leq \eta b + C_1 \log b,
\end{equation*}
so the first condition of the enumeration lemma is satisfied. To see that the second is also satisfied, for each $w\in B_{2^{-a}}(x)$ such that $w$ agrees with $x$ up to precision $s$, 
\begin{align*}
    K^{A, D}_b(w)&\geq K_s^{A, D}(x) + K_{b-s, b}^{A, D}(e\mid x) - O(\log b)\\
    &\geq K_s^{A, D}(x) + K_{b-s, b}^{A}(e\mid x) - O(\log b)\\
    &\geq K_s^{A, D}(x) + \sigma (b - s) - c \log b - O(\log b)\\
    &=\min \{\eta b, K_s^A(x) \}+ \sigma (b - s) - c \log b - O(\log b). 
\end{align*}
If the first term is the smaller, letting $C=\max \{C_1, C_2\}$
\begin{align*}
    K^{A, D}_b(w)&\geq \eta b+ \sigma (b - s) - (C_2+c)\log b\\
    &\geq \eta b+ \sigma (b - s) -  (C + c) \log b. 
\end{align*}
So both conditions of the enumeration lemma hold. If the second term is smaller, 
\begin{align*}
    K^{A, D}_b(w)&\geq K_s^A(x)+ \sigma (b - s) - O(\log b)\\
    &= K_a^A(x) +  K_{s, a}^A(x)+ \sigma (b - s) - O(\log b).
\end{align*}
Applying the second condition gives 
\begin{align*}
    K^{A, D}_b(w)&\geq K_a^A(x)+ \sigma (b - a) - O(\log b)\\
    &= \eta b + 2(\log b)^2 - O(\log r)\\
    &\geq \eta b + (\log b)^2
\end{align*}
\noindent and again both conditions of the enumeration lemma hold. Applying it yields
\begin{equation*}
K^{A, D}_{b, b, a}(x\mid e, x)\leq K^{A, D}_{b, b, a}(p_e \mid e, x) + (\frac{4(C+ c+1)}{\sigma} + C_{x, e})\log b + K^{A, D}(\eta, C+ c).
\end{equation*}
Collecting the terms, we have
\begin{align*}
K^{A, D}_{b, b, a}(x\mid e, x)&\leq K^{A, D}_{b, b, a}(p_e x\mid e, x) + \frac{4c}{\sigma}\log b + K^{A, D}(\eta, C+c) + O_{x, e, \sigma}(\log b)\\
&\leq K^{A, D}_{b, b, a}(p_e x\mid e, x)+ \frac{4c}{\sigma}\log b + K(\eta, C+c) + O_{x, e, \sigma}(\log b)\\
&\leq K^{A, D}_{b, b, a}(p_ex \mid e, x)+ \frac{4c}{\sigma}\log b + K(\eta) +  K(C+c) + O_{x, e, \sigma}(\log b)\\
&\leq K^{A, D}_{b, b, a}(p_ex \mid e, x)+ \frac{4c}{\sigma}\log b + K(\eta) + K(c) + O_{x, e, \sigma}(\log b).
\end{align*}
Now observe that, by definition, $\eta$ can be computed from $K^A_a(x)$, $\sigma$, $b$, and $a$. We may assume $b$ is large enough that $K^A_a(x)<3b$. Since $a<b$, with the above and symmetry of information, this implies $K(\eta) \leq O_\sigma(\log b)$. Furthermore, $c\leq \frac{b}{\log b}$ implies $K(c)\leq O(\log b)$, hence
\begin{equation}\label{eq:prelimSoI}
K^{A, D}_{b, b, a}(x\mid e, x)\leq K^{A, D}_{b, b, a}(p_e x\mid e, x) + \frac{4c}{\sigma}\log b+ O_{x, e, \sigma}(\log b).
\end{equation}
By symmetry of information, 
\begin{equation*}
    K^{A, D}_{b,b, b, a}(x\mid p_e x, e, x) = K^{A, D}_{b, b, a}(x\mid e, x) - K^{A, D}_{b, b, a}(p_e x \mid e, x) + O(\log b).
\end{equation*}
In conjunction with \eqref{eq:prelimSoI}, this implies 
\begin{equation}
   K^{A, D}_{b,b, b, a}(x\mid p_e x, e, x) \leq \frac{4c}{\sigma}\log b+ O_{x, e, \sigma}(\log b). 
\end{equation}
It remains to remove the oracle $D$. By the properties of $D$ and the definition of $\eta$, 
\begin{align*}
   K^{A}_{b,b, b, a}(x\mid p_e x, e, x) &\leq   K^{A, D}_{b,b, b, a}(x\mid p_e x, e, x) + K^A_b(x) - \eta b + O(\log b)\\
   &\leq  K^A_b(x) - \eta b + \frac{4c}{\sigma}\log b+ O_{x, e, \sigma}(\log b)\\ 
   &=  K^A_a(x) + K^A_{b, a}(x)  - \eta b + \frac{4c}{\sigma}\log b+ O_{x, e, \sigma}(\log b)\\
    &=  K^A_{b, a}(x) - \sigma(b-a) + 2 (\log b)^2+ \frac{4c}{\sigma}\log b+ O_{x, e, \sigma}(\log b).
\end{align*}
Hence, if $b$ is sufficiently large, 
\begin{equation*}
K^A_{b,b, b, a}(x\mid p_e x, e, x)\leq K^A_{b, a}(x) - \sigma(b - a) + \frac{4c}{\sigma}\log b +3 (\log b)^2.
\end{equation*}
\end{proof}

\noindent Similarly, we have the following lemma, corresponding to teal intervals.

\begin{lem}\label{lem:pseudoTeal}
    Let $x\in\mathbb{R}^2$, $e\in S^1$, $\sigma\in \mathbb{Q}\cap(0, 1]$, $c\in\mathbb{N}$, $A\subseteq\mathbb{N}$ and $a<b\in \mathbb{N}$. Assume $b$ is sufficiently large depending on $e$, $x$, and $\sigma$, and the following conditions are satisfied:
\begin{enumerate}
    \item $K_s^{A}(e\mid x)\geq \sigma s - c \log b$ for all $s\leq b-a$
    \item $K^A_{b, s}(x)\leq \sigma (b - s) + c \log b $ for all $a\leq s\leq b$
\end{enumerate}
Then
\begin{equation*}
K^A_{b,b, b, a}(x\mid p_e x, e, x)\leq \frac{4c}{\sigma}\log b +3 (\log b)^2
\end{equation*}
\end{lem}
We omit the proof, since once one picks $\eta$ such that $\eta b = K^A_b(x) - 2(\log b)^2$, the argument is very similar to that of Lemma \ref{lem:pseudoYellow}. With these two lemmas, we are in a position to prove Lemma \ref{lem:projectionYellowTeal}. 

\begin{proof}
    Lemma \ref{lem:projectionYellowTeal} assumes the first condition of Lemmas \ref{lem:pseudoYellow} and \ref{lem:pseudoTeal}, so it remains to check the second. Assume $b$ is sufficiently large and that $s\in[\lceil a\rceil, \lceil b\rceil]$. If $[a, b]$ is $(\sigma, c)$-yellow, then 
\begin{equation*}
    K^A_{s,\lceil a\rceil}(x\mid x) \geq K^A_{s, a}(x\mid x) + \log b \geq  \sigma(s-a) - (c+1) \log b,
\end{equation*}
    If $[a, b]$ is $(\sigma, c)$-teal, then  
\begin{equation*}
       K^A_{\lceil b\rceil, s}(x\mid x) \leq K^A_{b, s}(x\mid x)+ \log b\leq \sigma(b-s) + (c+1) \log b,
\end{equation*}
In both cases, we apply our the respective preliminary lemmas, and note that when we can assume $b$ is sufficiently large depending on $x, e$, and $\sigma$, we have that
\begin{equation*}
    \frac{4c}{\sigma}\log b +3 (\log b)^2\leq O_c(\log b)^2.
\end{equation*}
\end{proof}

\section{Small sets of directions}\label{sec:smallsets}
In this section, we will define a family of small sets of directions. These will be the universal sets described in the introduction.

We will fix the following rapidly increasing sequence of natural numbers. Let $r_1 = 2$, and inductively define $r_{n+1} = 2^{2^{r_n}}$. Define the set 
\begin{center}
    $\ML := \{\theta \in (0, \pi) \mid (\exists c) \; K_r(\theta) \geq r - c\log r \text{ for every } r\in\N\}$
\end{center}

We will now use construct sets of directions by modifying the bits of the angles in $\ML$.  We begin by defining $\mathcal{D}_0$. For each $\theta \in \ML$, let $d_\theta$ to be the real whose binary expansion is given by 
\begin{align*}
d_\theta[r] = \begin{cases}
0 &\text{ if } r_n < r \leq n r_n \text{ for some } n \in \N \\
\theta[r] &\text{ otherwise}
\end{cases}
\end{align*}
Let $\mathcal{D}_0 \subseteq \mathcal{S}^1$ be the set 
\begin{equation}
    \mathcal{D}_0 = \{ e \in \mathcal{S}^1 \mid e = (\cos d_\theta, \sin d_\theta), \; \theta \in \ML\}
\end{equation}

Let $A\subseteq\N$ be an oracle. We write $e \in \mathcal{D}_0(A)$ if there is a constant $c$ such that following conditions hold.
\begin{enumerate}
    \item For every $n\in\N$, and every $r \leq r_n$,
    \begin{center}
    $K^{A}_r(e) \geq r - c\log r_n$
\end{center}
    \item For every $n\in\N$ and every $nr_n \leq r \leq r_{n+1}$,
    \begin{center}
    $K^{A}_r(e) \geq r - (n-1)r_n - c\log r$.
\end{center}
\end{enumerate}

Our next lemma collects the facts about $\mathcal{D}_0$ we will need in the proof of Theorems \ref{thm:BourgainUniversal} and \ref{thm:universalSetsADRegular}.
\begin{lem}\label{lem:propertiesOfD0}
The lower box (and hence Hausdorff) dimension of $\mathcal{D}_0$ is zero. For every countable sequence of oracles $A_i\subseteq \N$, there is a direction $e\in \mathcal{D}_0$ such that $e \in \mathcal{D}_0(A_i)$ for every $i\in \N$.
\end{lem}
\begin{proof}
    To see that the lower box dimension of $\mathcal{D}_0$ is 0, note that the map from $d_\theta $ to $e$ is bi-Lipschitz. Hence, it suffices to show that the lower box dimension of $\{d_\theta: \theta\in (0, \pi)\}$ is 0. By the definition of $d_\theta$, this set can be covered by $2^{r_n}$ intervals of length $2^{-n r_n}$ for all $n$. Since $\frac{r_n}{n r_n}$ goes to 0, the lower box dimension is indeed 0. 
    
    Note that, for any $a\in (0, \pi)$, if $e \in \mathcal{S}^1$ such that $e = (\cos a, \sin a)$, then $K_r(e) = K_r(a) +O(\log r)$. Let $\{A_i\}$ be a countable sequence of oracles. Then there is an angle $\theta \in \ML$ and constants $c^\prime_i$, $i\in\N$, such that $K^A_r(\theta) \geq r-c^\prime_i\log r$ for every $r\in\N$. Let $e \in \mathcal{D}_0$ such that $e = (\cos d_\theta, \sin d_\theta)$. Fix an $i \in\N$, and let $r\in \N$. Suppose that $r \leq r_n$, for some $n\in\N$. Then we see that
    \begin{align*}
        K^{A_i}_r(e) &\geq K^{A_i}_r(d_\theta) - O(\log r)\\
        &\geq K^A_r(\theta) - O(\log r_n)\\
        &\geq r - O(\log r_n).
    \end{align*}
    Now suppose that $nr_n \leq r \leq r_{n+1}$. 
    \begin{align*}
        K^A_r(e) &\geq K^A_r(d_\theta) - O(\log r)\\
        &\geq K^A_{r, nr_n}(d_\theta) +K^A_{r_n}(d_\theta) - O(\log r)\\
        &\geq K^A_{r, nr_n}(\theta) +K^A_{r_n}(\theta) - O(\log r)\\
        &\geq r -(n-1)r_n - O(\log r),
    \end{align*}
    and the proof is complete.
\end{proof}

Let $0 < s < 1$. For each $\theta \in \ML$, let $d^s_\theta$ to be the real whose binary expansion is given by 
\begin{align*}
d^s_\theta[r] = \begin{cases}
0 &\text{ if } r_n < r \leq \lfloor r_n / s \rfloor \text{ for some } n \in \N \\
\theta[r] &\text{ otherwise}
\end{cases}
\end{align*}
Let $\mathcal{D}_s \subseteq \mathcal{S}^1$ be the set 
\begin{equation}
    \mathcal{D}_s = \{e = (\cos d^s_\theta,\sin d^s_\theta) \mid \theta \in \ML\}
\end{equation}

Let $A\subseteq\N$ be an oracle. We write $e \in \mathcal{D}_s(A)$ if there is a constant $c$ such that following conditions hold.
\begin{enumerate}
    \item For every $n\in\N$, and every $r \leq r_n$,
    \begin{center}
    $K^{A}_r(e) \geq r - c\log r_n$
\end{center}
    \item For every $n\in\N$ and every $\lfloor r_n/s\rfloor \leq r \leq r_{n+1}$,
    \begin{center}
    $K^{A}_r(e) \geq r - \lfloor \frac{1-s}{s}\rfloor r_n - c\log r$.
\end{center}
\end{enumerate}
Our next lemma collects the properties of $\mathcal{D}_s$ we will need in the proofs of Theorem \ref{thm:CountablyUniversalWeaklyRegular}. 
\begin{lem}\label{lem:propertiesOfDs}
    The lower box dimension of $\mathcal{D}_s$ is $s$. For every countable sequence of oracles $A_i\subseteq \N$, there is a direction $e\in \mathcal{D}_s$ such that $e \in \mathcal{D}_s(A_i)$ for every $i\in \N$.
\end{lem}
\begin{proof}
    
Repeating the first portion of the argument in Lemma \ref{lem:propertiesOfD0} implies that this set has lower box dimension at most $s$. To see the lower box dimension is exactly $s$, it suffices to lower bound the Hausdorff dimension by $s$. let $A\subseteq \N$ be an oracle. Let $\theta \in ML$ and $c>0$ such that $K^A_r(\theta) > r -c\log r$ for every $r\in\N$. Note that one exists. Let $e \in \mathcal{D}_s$ such that $e = (\cos d^s_\theta, \sin d^s_\theta)$. The same argument as the previous paragraph shows that 
    \begin{align*}
         K^A_{r_n}(e) &= K_{r_n}(\theta) + O(\log(r_n))\\
         &= r_n + O(\log r_n),
    \end{align*}
    for every $n\in\N$. Let $r\in\N$ be sufficiently large. We first suppose that $r_n \leq r < r_n/s$. Then we see from the above equality that $K^A_{r}(e) \geq r_n + O(\log r_n)$, and so $ K^A_r(e) \geq sr - o(r)$.

    If $r_n/s \leq r < r_{n+1}$, then 
    \begin{align*}
        K^A_r(e) &\geq K^A_{r, r_n/s}(\theta) + K^A_{r_n}(\theta) - O(\log r)\\
        &= r - \left(\frac{1-s}{s}\right)r_n,
    \end{align*}
   and so $ K^A_r(e) \geq sr - o(r)$ in this case as well.

    Let $\{A_i\}$ be a countable sequence of oracles. Then there is an angle $\theta \in \ML$ and constants $c^\prime_i$, $i\in\N$, such that $K^A_r(\theta) \geq r-c^\prime_i\log r$ for every $r\in\N$. Let $e \in \mathcal{D}_s$ such that $e = (\cos d^s_\theta, \sin d^s_\theta)$. Fix an $i \in\N$, and let $r\in \N$. Suppose that $r \leq r_n$, for some $n\in\N$. Then we see that
    \begin{align*}
        K^{A_i}_r(e) &\geq K^{A_i}_r(d_\theta) - O(\log r)\\
        &\geq K^A_r(\theta) - O(\log r_n)\\
        &\geq r - O(\log r_n).
    \end{align*}
    Now suppose that $r_n/s \leq r \leq r_{n+1}$. Then
    \begin{align*}
        K^A_r(e) &\geq K^A_r(d^s_\theta) - O(\log r)\\
        &\geq K^A_{r, r_n/s}(\theta) +K^A_{r_n}(\theta) - O(\log r)\\
        &\geq r -(n-1)r_n - O(\log r),
    \end{align*}
    and the proof is complete.
\end{proof}

\section{Universal sets under regularity assumptions}\label{sec:UniversalSetsRegularity}
In this section, we prove the main results of the paper. Namely, we show that assuming sufficiently regularity of the class gives smaller universal sets of directions for that class.

\subsection{Weakly regular sets}\label{ssec:WeaklyRegular}
We begin with the class of weakly regular sets. Recall that a set $E\subseteq\R^2$ is weakly regular if $\dim_H(E) = \dim_H(P)$. 

As we are using the machinery of effective dimension, our main theorem follows from its point-wise analog.
\begin{prop}\label{prop:effThmWeaklyRegular}
Let $A,B\subseteq\N$, $x\in \R^2$, $e\in \mathcal{S}^1$, $\alpha \in (0,2]$, $0 < s < 1$ and $\ve > 0$. Assume that the following conditions hold.
\begin{enumerate}
    \item $\alpha - \ve \leq \dim^A(x) \leq \Dim^A(x)\leq \alpha$.
    \item $e \in \mathcal{D}_s(A)$.
    \item $K^{A, B, e}_s(x) \geq K^A_r(x) -\ve r$ for every sufficiently large $r\in\N$.
\end{enumerate}
Then 
\begin{center}
    $\dim^{A,B,e}(p_e x) \geq \min\{\alpha, 1\} - 25\ve r$
\end{center}
\end{prop}
\begin{proof}
    We first note that, by condition (3) and the symmetry of information, 
    \begin{align*}
        K^A_r(e\mid x) &= K^A_r(e) + K^A_{r}(x\mid e) - K^A_r(x) - O(\log r)\\
        &\geq K^A_r(e) + K^{A,e}_{r}(x) - K^A_r(x) - O(\log r)\\
        &\geq K^A_r(e) - \ve r - O(\log r),
    \end{align*}
    for all sufficiently large $r$.
    
    Let $r\in\N$ be sufficiently large. Recalling the sequence defined in \ref{sec:smallsets}, we first assume that $r_n \leq r < \frac{1}{\ve^2}\lfloor r_n / s \rfloor$. Let $M = \lfloor r/r_n\rfloor$, and define the sequence $r_i = ir_n$ for $i = 1, \ldots, M$, and $r_{M+1} = r$. Note that $[r_i, r_{i+1}]$ is both $(\alpha, \ve)$-almost teal and $(\alpha, \ve)$-almost yellow. 
    
    Using symmetry of information, and Lemma \ref{lem:almostProjectionYellowTeal}, with $\sigma = \min\{\alpha, 1\}$ we conclude that 
    \begin{align*}
        K^{A, B, e}_r(x \mid p_e x) &\leq \sum\limits_{i=1}^{M} K^{A, e}_{r_{i+1}, r_{i+1}, r_i}(x \mid p_e x, e) + O(\log r_i)\\
        &\leq K^A_r(x) - \sigma r+ 4M\ve r+ CM\log r,
    \end{align*}
    for some fixed constant $C$.

    Using the symmetry of information to rearrange the above inequality and using condition (3) of the present proposition, we have
    \begin{align*}
        K^{A,B,e}_r(p_e x) &\geq K^{A, B, e}_r(x) - \left(K^A_r(x) - \sigma r+ 4M\ve r+ O(log r)\right)\\
        &\geq \sigma r - (4M+1)\ve r - O(\log r)\\
        &= \min\{\alpha, 1\} - (4M+1)\ve r - O(\log r).
    \end{align*}
    By the definition of $M$ and our assumption on $r$, we conclude that
    \begin{equation}
         K^{A,B,e}_r(p_e x) \geq \min\{\alpha, 1\} - 5\ve r - O(\log r).
    \end{equation}

    For the second case, assume that $r \geq \frac{1}{\ve^2}\lfloor r_n / s \rfloor$. Let $t \leq r$. If $t \leq r_n$, then we immediately have $K^A_t(e) \geq t - O(\log r)$. If $r_n < t \leq r_n/s$, then 
        \begin{align*}
            K^A_t(e) &\geq r_n - O(\log r_n)\\
            &> t - \ve^2 r - O(\log r).
        \end{align*}
        Finally, if $r_n/s < t \leq r$,
        \begin{align*}
            K^A_t(e) &\geq K^A_{t, r_n/s}(e\mid e) + K^A_{r_n}(e) - O(\log t)\\
            &= t - \frac{r_n}{s} + r_n - O(\log t)\\
            &= t - \ve^2 r - O(\log t).
        \end{align*}
        Therefore, we see that, for all $t\leq r$, 
        \begin{equation}
            K^A_t(e) \geq t - \ve^2 r - O(\log r)
        \end{equation}
        It follows from Theorem \ref{thm:LutStu22} that 
        \begin{center}
            $K^{A,B,e}_r(p_e x) \geq \min\{r, K^A_r(x)\} - 25\ve r - O(\log r)$.
        \end{center}
         
        We have now established this bound for all $r \in [r_n, r_{n+1})$; the conclusion follows from taking the limit inferior.
\end{proof}

We now use the previous, pointwise, result and the point-to-set principle to establish our main theorem for $\mathcal{C}_{WR}$, the class of weakly regular sets.
\begin{thm}\label{thm:CountablyUniversalWeaklyRegular}
    For every $\ve > 0$, there is a set of directions of lower box counting dimension $\ve$ which is universal for $\mathcal{C}_{WR}$. In particular, for every $0 < s < 1$, the set $\mathcal{D}_s$ is universal for $\mathcal{C}_{WR}$.
\end{thm}
\begin{proof}
    Fix an $s \in (0,1)$. Let $E$ be a weakly regular subset of $\R^2$. Let $\alpha = \dim_H(E) = \dim_P(E)$. Note that, since $E$ is weakly regular, it necessarily has an optimal oracle $A$. Without loss of generality, we assume that $A$ is also a packing oracle for $E$. Using Lemma \ref{lem:propertiesOfDs}, choose $e \in D_s$ to be a direction such that $e \in \mathcal{D}_s(A)$. 

    Let $B$ be a Hausdorff oracle for $p_e E$. To complete the proof, it suffices to show that, for every $\ve > 0$, there is a point $x\in E$ such that
    \begin{equation}
        \dim^{A, B, e}(p_e x) \geq \min\{\dim_H(E), 1\} - \ve.
    \end{equation}

    Fix $\ve > 0$. Since $A$ is optimal for $E$, there is a point $x\in E$ such that $\dim^{A, B, e}(x) > \dim_H(E) - \frac{\ve}{100}r$ and
    \begin{equation*}
        K^{A, B, e}_r(x)\geq K^{A}_r(x) -\frac{\ve}{100}r
    \end{equation*}
    for all sufficiently large $r\in\N$. We also note that, since $A$ is a packing oracle for $E$, $\Dim^{A}(x) \leq \dim_P(E)$. By our choice of $e$ and $x$, we see that the conditions of Proposition \ref{prop:effThmWeaklyRegular} hold, and so we have
    \begin{center}
        $\dim^{A, B, e}(p_e x) \geq \min\{\dim_H(E), 1\} - \ve$,
    \end{center}
    completing the proof.
\end{proof}

\subsection{AD-regular sets}\label{ssec:ADRegular}
We now establish the existence of dimension zero universal sets for the class of AD-regular sets. We begin by recalling the relevant definitions.

Let $E \subseteq \R^n$ be a closed set, $C\geq 1$ and $\alpha \geq 0$. We say that $E$ is $(\alpha, C)$-AD-regular if
\begin{equation}
    C^{-1}r^\alpha \leq \mathcal{H}^\alpha(E \cap B(x, r))\leq Cr^\alpha, \;\;\; x \in A, 0 < r < \text{diam}(A)
\end{equation}
We say that $E$ is $\alpha$-AD-regular if $E$ is $(\alpha, C)$-AD-regular, for some constant $C \geq 1$.

As in \cite{FieStu24}, we can define a point-wise analog of AD-regular sets. A point $x\in \R^n$ is $(\alpha, C)$-AD-regular with respect to an oracle $A\subseteq \N$ if
\begin{equation}
     \alpha r - C\log r \leq K_r^A(x)\leq  \alpha r + C\log r 
\end{equation}
for every $r\in \N$.

Let $E\subseteq \R^n$ be $\alpha$-AD-regular. We say that an oracle $A \subseteq\N$ is an AD-regular oracle for $E$ if 
\begin{equation}
     \alpha r - O(\log r) \leq K_r^A(x)\leq  \alpha r + O(\log r)  \; \forall r \in \N
\end{equation}
for $\mathcal{H}^\alpha$-a.e. $x\in E$. Note that we immediately have that any AD-regular oracle for $E$ is a Hausdorff oracle for $E$. The next lemma, from \cite{FieStu24}, gives the existence of AD-regular oracles for AD-regular sets, and can thus be interpreted as a sort of point-to-set principle for AD-regularity.\footnote{This lemma was originally stated for compact AD-regular sets, which immediately entails it for closed AD-regular sets. Here, we adopt what seems to be the more common convention that AD-regular sets are assumed to be closed by definition.} 

\begin{lem}\label{lem:ADRegularOraclesExist}
Let $E\subseteq\mathbb{R}^n$ be an $\alpha$-AD-regular set. Then there is an AD-regular oracle $A$ for $E$. Moreover, for any oracle $B\subseteq\mathbb{N}$, the join $(A,B)$ of $A$ and $B$ is also an AD-regular oracle for $E$.
\end{lem}

As before, we will establish our main result by first proving its point-wise analog.
\begin{prop}\label{prop:effThmADRegular}
Let $A\subseteq \N$, $x\in \R^2$, $e \in \mathcal{S}^1$, $\alpha \in (0, 2]$ and $C \in \N$. Suppose that $x$ is $(\alpha, C)$-AD-regular relative to $A$, and $K^{A, e}_r(x) \geq K^A_r(x) - C\log r$ for every $r\in\N$ and $e\in\mathcal{D}_0(A)$. Then
\begin{center}
    $\dim^{A,e}(p_e x) \geq \min\{\alpha, 1\}$.
\end{center}
\end{prop}
\begin{proof}
    Let $A, x, e, \alpha$ and $C$ satisfy the conditions. Let $\ve > 0$ and let $r\in\N$ be sufficiently large. Let $n\in\N$ such that $r_n \leq r < r_{n+1}$. There are two cases to consider. For the first, assume that $r \leq (nr_n)^2$. Let $M = \lfloor r / r_n \rfloor$, and define the sequence $r_i = i r_n$ for every $i = 1,\ldots M$ and $r_{M+1} = r$. Since $x$ is $(\alpha, C)$-AD-regular relative to $A$ we have that $[r_i, r_{i+1}]$ is $(\alpha, C)$-yellow. Hence, we may apply Lemma \ref{lem:projectionYellowTeal}, which implies
    \begin{align*}
        K^{A, e}_r(x \mid p_e x) &\leq \sum\limits_{i=1}^{M+1} K^{A, e}_{r_{i+1}, r_{i+1}, r_i}(x \mid p_e x, e) + O(\log r_i)\\
        &\leq \min\{\alpha - 1, 0\} r + O(M \log r).
    \end{align*}
    By the assumption that $K^{A, e}_r(x) \geq K^A_r(x) - C\log r$ for every $r\in\N$, we have
    \begin{align*}
        \min\{\alpha - 1, 0\} r + O(\log r) &= K^{A, e}_r(x \mid p_e x)\\
        &= K^{A, e}_r(x) - K^{A, e}_r(p_e x) - O(\log r)\\
        &= \alpha r - K^{A, e}_r(p_e x) - O(M \log r).
    \end{align*}
    Rearranging, and applying the definition of $M$ we conclude that 
    \begin{equation}
        K^{A,e}_r(p_e x) \geq \min\{\alpha, 1\}r - o(r).
    \end{equation}
    
    For the second case, assume that $r \geq (nr_n)^2$. Note that, by assumption $K^A_t(e) \geq t - O(\log r)$ for all $t\leq r_n$. For every $r_n < t \leq r$, 
    \begin{align*}
        K^A_t(e) &\geq t - (n-1)r_n - O(\log r)\\
        &\geq t - \frac{\ve}{100}r - O(\log r).
    \end{align*}
    Hence, by Theorem \ref{thm:LutStu22}, we conclude that 
    \begin{center}
        $K^{A,e}_r(p_e x) \geq \min\{\alpha, 1\} r - \ve r$.
    \end{center}
    Since $\ve$ can be chosen to be arbitrarily small in the second case, the conclusion follows.
\end{proof}

We are now able to prove our main result on the existence of small universal sets of directions for AD-regular sets. 
\begin{thm}\label{thm:universalSetsADRegular}
There is a set of directions of lower box counting dimension zero which is universal for $\mathcal{C}_{AD}$. In particular, the set $\mathcal{D}_0$ is a universal set of directions for the class $\mathcal{C}_{AD}$.
\end{thm}
\begin{proof}
    Let $E$ be an AD-regular subset of $\R^2$. Let $A$ be an AD-regular oracle for $E$, guaranteed by Lemma \ref{lem:ADRegularOraclesExist}. We may assume that $E$ is computably compact relative to $A$. Using Lemma \ref{lem:propertiesOfD0}, choose $e \in \mathcal{D}_0$ to be a direction such that $e \in \mathcal{D}_0(A)$. 

    Since each $E$ is computably compact relative to $A$, the oracle $(A, e)$ is a Hausdorff oracle for $p_e E$. To complete the proof, it suffices to show that, for every $\ve > 0$, there is a point $x\in E$ such that
    \begin{equation}
        \dim^{A, e}(p_e x) \geq \min\{\dim_H(E), 1\} - \ve.
    \end{equation}

    To this end, fix $\ve > 0$. Since $A$ is an AD-regular oracle for $E$, there is a point $x\in E$ such that $\dim^{A, e}(x) = \dim_H(E)$ and
    \begin{equation*}
        K^{A, e}_r(x)\geq K^{A}_r(x) -O(\log r),
    \end{equation*}
    for all sufficiently large $r\in\N$. By our choice of $e$ and $x$, it is easy to verify that the conditions of Proposition \ref{prop:effThmADRegular} hold, and so we have
    \begin{center}
        $\dim^{A, e}(p_e x) \geq \min\{\dim_H(E), 1\} - \ve$,
    \end{center}
    completing the proof.
\end{proof}

\begin{remark}
    A nearly identical argument allows us to establish Proposition \ref{prop:packingRegular}. We omit the details but outline the modifications. It is straightforward to modestly generalize the constructions in Section 4 by defining 
    \begin{center}
    $\ML(s) := \{\theta \in (0, \pi) \mid (\exists c) \; s r + c\log r \geq K_r(\theta) \geq s r - c\log r \text{ for every } r\in\N\}$
\end{center}
and considering the corresponding sets $\mathcal{D}_{\ve, s}(A)$. Applying the point to set principle, these sets have packing dimension $s$. If $E$ is (weakly or AD) regular and has dimension no more than $s<1$, then the maximum complexity growth rate of $x\in E$ on any interval is essentially $s$, so we can partition $[1, r]$ into many $s$-teal intervals. Since we want to use $s$-teal intervals and not $1$-teal intervals, Lemmas \ref{lem:projectionYellowTeal} and \ref{lem:alternateProjectionYellowTeal} only require that the directions have complexity at least $s r$. Hence, we can choose $e\in \mathcal{D}_{\ve, s}(A)$ and still obtain the optimal bounds after partitioning. 
\end{remark}

\section{Bourgain universal sets}
In this section, we prove that there are Bourgain universal sets of lower box counting dimension zero, for the class of analytic sets. Recall that $\mathcal{D}\subseteq\mathcal{S}^1$ is Bourgain universal for the class of analytic sets if, for every analytic $E\subseteq\R^2$, there is a direction $e\in \mathcal{D}$ such that
\begin{center}
    $\dim_H(p_e E) \geq \frac{\dim_H(E)}{2}$.
\end{center}

We begin with the point-wise analog of this theorem. 
\begin{prop}\label{prop:effThmBourgain}
    Let $A,B\subseteq\N$, $x\in\R^2$, $e\in\mathcal{S}^1$, and $\ve > 0$. Suppose that $\ve$ is sufficiently small, and the following conditions hold.
    \begin{enumerate}
        \item $e \in \mathcal{D}_0(A)$.
        \item $K^{A,B, e}_r(x) \geq K^A_r(x) - c\log r$ for some fixed constant $c > 0$ and all sufficiently large $r\in\N$.
    \end{enumerate}
    Then,
    \begin{equation}
        \dim^{A,B,e}(p_e x) \geq \frac{\dim^A(x)}{2}.
    \end{equation}
\end{prop}

\begin{proof}
Let $r$ be sufficiently large. To use Lemmas \ref{lem:projectionYellowTeal} and \ref{lem:alternateProjectionYellowTeal}, we need to ensure that, for all $b\in[1, r]$, $K_{s, b}^A(e\mid x)$ is large. That is, we must show that the information in $x$ does not help in computing $e$, whereas condition two indicates that the information in $e$ is not useful in computing $x$. Using this condition and the symmetry of information, 
\begin{align*}
    K_{s, r}^A(e\mid x)
    &=K^A_{s}(e) + K^{A}_{r,s}(x\mid e) - K^A_{r}(x) - O(\log r)\\
    &\geq K^A_{s}(e) + K^{A,e}_{r}(x) - K^A_{r}(x) - O(\log r)\\
    &\geq K^A_s(e) - O(\log r).
\end{align*}
 Since 
\begin{equation*}
    K_{s, b}^A(e\mid x)\geq K_{s, r}^A(e\mid x) - O(\log r),
\end{equation*}
we see that
\begin{equation}\label{eq:xWontHelpe}
    K_{s, b}^A(e\mid x)\geq K^A_s(e) - O(\log r).
\end{equation}
Recalling the sequence of precisions defined in Section 4, let $n$ be such that $r_n < r\leq r_{n+1}$. There are two cases to consider. 

First, assume $r\in[r_n, 4 n^2 r_n^2]$. By the definition of $D_0(A)$, for all $s<r_n$, 
$K^A_s(e)\geq s - O(\log r_n)$. Applying \eqref{eq:xWontHelpe}, for some $k$ independent of $r$
\begin{equation*}
    K^A_{s, r_n}(e\mid x)\geq s - k\log r_n.
\end{equation*}
We may assume $r$ is large enough that $\log r$ is sufficiently large (as in Lemma \ref{lem:projectionYellowTeal}). We may thus apply Lemma \ref{lem:projectionYellowTeal} on all $(1, k)$-yellow and $(1, k)$-teal intervals $[a, b]$ such that $\log r<a$ and $b-a<r_n$. By Lemma 17 in \cite{Stull22c}, we can partition $[\log r, r]$ into $M\leq 3\lceil \frac{r}{r_n}\rceil$ such intervals; label this sequence $\log r = a_0, a_1, ... a_M=r$ and write $i\in Y$ if $[a_{i-1}, a_i]$ is $(1, k)$-yellow. Let $L$ denote the total length of the yellow intervals. Then
\begin{align*}
    K^A_r(x\mid e, p_e x)&=K^A_{r, r, r, \log r}(x\mid e, p_e x, x) + K^A_{\log r, r, r}(x\mid e, p_e x) + O(\log r)\\
    &\leq \sum_{i=1}^M \left(K^A_{a_i, a_i, a_i, a_{i-1}}(x\mid e, p_e x, x) + O(\log r) \right) + O(\log r)\\
    &\leq \sum_{i\in Y} K^A_{a_i,a_i, a_i, a_{i-1}}(x\mid e, p_e x, x) + \sum_{i\notin Y} K^A_{a_i, a_i, a_i, a_{i-1}}(x\mid e, p_e x, x) + O(M \log r)\\ &\leq \sum_{i\in Y} K^A_{a_i, a_{i-1}}(x) -  L + O(M (\log r)^2)\\
    &\leq \sum_{i\in Y} K^A_{a_i, a_{i-1}}(x) -  L + O(\frac{r}{r_n} (\log r)^2)\\
    &\leq \sum_{i\in Y} K^A_{a_i, a_{i-1}}(x) -  L + O(\frac{n r}{\sqrt r} (\log r)^2).
\end{align*}
Since the sequence $r_n$ increases rapidly, this implies
\begin{equation}\label{eq:effectiveBourgainPartition}
K^A_r(x\mid e, p_e x)\leq \sum_{i\in Y} K^A_{a_i, a_{i-1}}(x) -  L + O(\sqrt{r} (\log r)^3).
\end{equation}

\noindent Now we consider the case that $r>  4 n^2 r_n^2$. Since $e\in D_0(A)$,
\begin{align*}
    K^A_r(e) &\geq r - (n-1) r_n - c\log r \\
    &\geq r - \frac{\sqrt{r}}{2} - c\log r. 
\end{align*}
Applying \eqref{eq:xWontHelpe}, we obtain that for all $s\leq r$, 
\begin{align*}
    K^A_{r}(e\mid x)&\geq s - \frac{\sqrt{r}}{2} - k \log r\\ 
    &\geq s - \sqrt{r}
\end{align*}
since $r$ is assumed to be sufficiently large. Applying Lemma 17 in \cite{Stull22c} (in the case $t=r$), there exists some $a_1$ such that $\log r=a_0\leq a_1\leq a_2=r$, $[a_0, a_1]$ is $(1, k)$-teal, and $[a_1, a_2]$ is $(1, k)$-yellow. Since each of these intervals necessarily has length less than $r$, in conjunction with the above inequality, we may apply Lemma \ref{lem:alternateProjectionYellowTeal} on each interval.
\begin{align*}
K_{r}^A(x\mid p_e x, e)&\leq K_{a_1, a_1, a_1, a_0}^A(x\mid p_e x, e, x) + K_{a_2, a_2, a_2, a_1}^A(x\mid p_e, e, x) + O(\log r)\\
&\leq O(\sqrt{a_1}) + K^A_{a_2, a_1}(x) - \sigma(a_2 - a_1) +O(\sqrt{r}) + O(\log r)\\
&= K^A_{a_2, a_1}(x) - (1 - 2\ve)(a_2 - a_1) +O(\sqrt{r}).
\end{align*}
Again letting $Y$ denote the set of yellow intervals and $L$ the length of these intervals, this implies
\begin{equation*}
K_{r}^A(x\mid p_e x, e)\leq  \sum_{i\in Y} K^A_{a_i, a_{i-1}}(x) - L + O(\sqrt{r} (\log r)^3),
\end{equation*}
which when combined with \eqref{eq:effectiveBourgainPartition} verifies that this bound holds for \emph{all} sufficiently large $r$. Now, observe that
\begin{align*}
    K^A_r(x\mid e, p_e x)&\geq  K^{A, B, e}_r(x\mid p_e x) - O(\log r)\\
    &=K^{A, B, e}_r(x, p_e x) -K^{A, B, e}_r(p_e x) - O(\log r)\\
    &\geq K^{A, B, e}_r(x) -K^{A, B, e}_r(p_e x) - O(\log r)\\
    &\geq K^{A}_r(x) -K^{A, B, e}_r(p_e x) - O(\log r).
\end{align*}
Hence, for all sufficiently large $r$, 
\begin{equation*}
K^{A, B, e}_r(p_e x)\geq K^{A}_r(x) - \sum_{i\in Y} K^A_{a_i, a_{i-1}}(x) +  L  - O(\sqrt{r} (\log r)^3).
\end{equation*}

For the remainder of the argument, we can freely apply symmetry of information, partitioning into at most $M$ intervals on each step. Each partition thus accumulates an error of $O(M \log r)$, but we have already seen this is dominated by $\sqrt{r} (\log r)^3$. Continuing, 

\begin{align*}
K^{A, B, e}_r(p_e x)&\geq \sum_{i\notin Y} K^A_{a_i, a_{i-1}}(x) +  L  - O(\sqrt{r} (\log r)^3)\\
&\geq \sum_{i\notin Y} \frac{K^A_{a_i, a_{i-1}}(x)}{2} +  L  - O(\sqrt{r} (\log r)^3).
\end{align*}
By \eqref{eq:caseLutz}, the complexity growth rate on each yellow interval is essentially bounded by $2$ up to a logarithmic term in $r$, hence
\begin{align*}
K^{A, B, e}_r(p_e x)&\geq \sum_{i\notin Y} \frac{K^A_{a_i, a_{i-1}}(x)}{2} + \sum_{i\in Y}\frac{K^A_{a_i, a_{i-1}}(x)}{2}  - O(\sqrt{r} (\log r)^3)\\
&\geq\frac{K^A_r(x)}{2}-  O(\sqrt{r} (\log r)^3)
\end{align*}
Taking the limit inferior on both sides completes the proof.

\end{proof}

Using the previous proposition, the point-to-set principle and standard facts from geometric measure theory, we conclude the main theorem of this section.
\begin{thm}\label{thm:BourgainUniversal}
    There is a set of directions of lower box counting dimension zero which is Bourgain universal for the class of analytic sets. In particular, the set $\mathcal{D}_0$ is Bourgain universal for the class of analytic sets. 
\end{thm}

\begin{proof}
    Let $E$ be an analytic set of Hausdorff dimension $s$. Since Hausdorff measures are inner regular for analytic sets, there is a sequence $E_i\subseteq E$ of compact sets such that $\dim_H(E_i) = s-\frac{1}{i}$ and $0<\mathcal{H}^{s-\frac{1}{i}}(E_i)<\infty$. Let $A_i$ be an oracle relative to which $E_i$ is effectively compact and relative to which he measure $\mu_i = \mathcal{H}^{s-\frac{1}{i}}\vert_{E_i}$ is computable. Let $A$ denote the join of these oracles, observing that every $E_i$ is effectively compact relative to $A$. Using Lemma \ref{lem:propertiesOfD0}, choose $e \in \mathcal{D}_0$ to be a direction such that $e \in \mathcal{D}_0(A)$. Note that by \cite{LutStu24}, $(A, e)$ is a Hausdorff oracle for $p_e E_i$, so by the point to set principle, for any $x\in E_i$
    \begin{equation}\label{eq:BourgainReductionBound}
        \dimH(p_e E)> \dim_H(p_e E_i) \geq \dim^{A,e}(p_e x)
    \end{equation}
     To complete the proof, we must lower bound the last term using Proposition \ref{prop:effThmBourgain}. To this end, fix an $i\in\N$. In our application, we do not need $B$ to contain any extra information, so it can be taken to be any oracle computable from $A$ and $e$. By our choice of $e$, the first condition holds, so it remains to show that there exists some $x\in E_i$ and $c$ such that the following inequality holds when $r$ is sufficiently large:
\begin{equation*}
K_r^{A, e}(x)\geq K_r^A(x) - c\log r
\end{equation*}
The smaller error term than in our other effective theorems implies we will need to use somewhat more sophisticated machinery to find such an $x$. We will refine $E_i$, first letting
\begin{equation*}
N_i = \{x\in E_i: \exists^\infty r\in \mathbb{N} \text{ satisfying } K^{A, e}(x) < K^A(x) - 8 \log r\}
\end{equation*}
By Lemma 36 of \cite{Stull22c}, since $E_i$ is effectively compact relative to $A$ and $\mu_i$ is computable relative to $A$, this set has $H^{s-\frac{1}{i}}$ measure 0. Hence, for any $x\in E_i \setminus N_i$, the conditions to apply Proposition \ref{prop:effThmBourgain} are satisfied. Finally, note that since $E_i$ has positive $\mathcal{H}^{s - \frac{1}{i}}$-measure, $E_i \setminus N_i$ has Hausdorff dimension $s-\frac{1}{i}$. Hence, by the point-to-set principle, we can choose $x\in E_i\setminus N_i$ to have effective Hausdorff dimension relative to $A$ at least $s - \frac{2}{i}$. Applying Proposition \ref{prop:effThmBourgain} and \eqref{eq:BourgainReductionBound} gives
\begin{align*}
 \dimH(p_e E)&> \dim^{A,e}(p_e x)\\
 &\geq \frac{\dim^{A}(x)}{2}\\
 &= \frac{s}{2} - \frac{1}{i}
 \end{align*}
 Since $i$ was arbitrary, this implies
 \begin{equation}
 \dimH(p_e E)\geq \frac{\dimH(E)}{2}
 \end{equation}
which completes the proof.

\end{proof}

\begin{remark}
    Proposition \ref{prop:OOBourgain} follows from a very similar argument. In the modified proof, one would use the optimal oracles property to establish a weaker version of the second condition in the effective theorem, namely
    \begin{equation*}
        K^{A,B, e}_r(x) \geq K^A_r(x) - \ve r
    \end{equation*}
    This would in turn imply
    \begin{equation*}
        K_{s, b}^A(e\mid x)\geq K^A_s(e) - \ve r -  O(\log r).   
    \end{equation*}
    The worse error term above implies the low complexity intervals of $e$ need to be shorter to apply Lemma \ref{lem:projectionYellowTeal} and consequently obtain the bound of $\dim^A(p_e x)\geq \frac{\dim^A(x)}{2}$ with our proof strategy. Therefore this approaches gives the desired effective projection theorem if $e\in \mathcal{D}_\ve(A)$ but not if $e\in \mathcal{D}_0(A)$. Finally, when applying the effective theorem, let $B$ be a Hausdorff oracle for $p_e E$. We consider it an interesting problem to either generalize Theorem \ref{thm:BourgainUniversal} to the class of sets with optimal oracles, or show that it fails for this larger class. 
\end{remark}

\section{Sets with optimal oracles}

In this section we show that Kaufman's exceptional set estimate generalizes to sets with optimal oracles (defined precisely at the end of Section 2). As described in the introduction, this immediately implies the existence of small universal sets for the class $\mathcal{C}_{OO}$.
\begin{thm}\label{thm:exceptionalsetOptimalOracles}
    Let $E\subseteq\R^2$ be a set with optimal oracles. Let $0 < s \leq \min\{\dim_H(E), 1\}$. Then,
    \begin{center}
        $\dim_H(\{e\in\mathcal{S}^1\mid \dim_H(p_e E) < s\}) \leq s$.
    \end{center}
\end{thm}
\begin{proof}
    We may assume that $s < 1$, since the claim is trivially true otherwise. Let $A_1$ be an optimal oracle for $E$. Suppose that the conclusion is false, i.e., $\dim_H(D_s(E)) > s$, where
    \begin{center}
        $D_s(E) := \{e\in\mathcal{S}^1\mid \dim_H(p_e E) < s\}$.
    \end{center}
    Let $A_2\subseteq\N$ be a Hausdorff oracle for $D_s(E)$. Let $e \in D_s(E)$ such that $\dim^{A_1,A_2}(e) > s$. Let $A_3$ be a Hausdorff oracle for $p_e E$, and let $A = (A_1, A_2, A_3)$. It suffices to show that 
    \begin{equation}
        \sup\limits_{x\in E}\dim^A(p_e x) \geq s.
    \end{equation}
    
    Since $A_1$ is an optimal oracle for $E$, there is a point $x\in E$ such that $\dim^A(x) > \dim_H(E) - \ve$ and 
    \begin{center}
        $K^{A, e}_r(x) > K^{A_1}_r(x) - \ve r$.
    \end{center}
    It is easy to verify that the conditions of Theorem \ref{thm:LutStu22} hold, and so we have
    \begin{center}
        $\dim^{A,e}(p_e x) \geq s - \ve$,
    \end{center}
    and the conclusion follows.
\end{proof}

As a consequence, we can prove the existence of small universal sets of the class of sets with optimal oracles. For every $0 < s \leq 1$, let $\mathcal{C}_{OO}(s)$ be the class of subsets of $\R^2$ with optimal oracles. We will use Theorem \ref{thm:exceptionalsetOptimalOracles} to prove that any essentially $s$-dimensional set of directions is universal for $\mathcal{C}_{OO}$. 

Let $n \geq 1$, and $s \in (0, n]$. We say that $E\subseteq \R^n$ is \textbf{essentially $s$-dimensional} if $\dim_H(E) = s$ and $E$ cannot be written as the countable union of sets $\{E_i\}$ such that $E_i \subseteq E$ and $\dim_H(E_i) < s$.

\begin{cor}
    The following are true.
    \begin{enumerate}
        \item If $\mathcal{D}\subseteq\mathcal{S}^1$ is an essentially $1$-dimensional set of directions, then $\mathcal{D}$ is universal for $\mathcal{C}_{OO}$.
        \item If $\mathcal{D}\subseteq\mathcal{S}^1$ is an essentially $s$-dimensional set of directions, then $\mathcal{D}$ is universal for $\mathcal{C}_{OO}(s)$.
    \end{enumerate}
\end{cor}

\begin{proof}
    We prove (2), the proof of (1) is nearly identical. Let $0 < s\leq 1$. Let $\mathcal{D}\subseteq\mathcal{S}^1$ be an essentially $s$-dimensional set of directions. Suppose that $\mathcal{D}$ is not universal for $\mathcal{C}_{OO}$, so there is a set $E\in\mathcal{C}_{OO}$ such that $\dim_H(p_e E) < \min\{s, \dim_H(E)\}$ for every $e\in \mathcal{D}$. Let $s^\prime := \min\{s, \dim_H(E)\}$. Then, 
    \begin{align*}
        \mathcal{D} &\subseteq \bigcup\limits_{i\in\N}\{e\in\mathcal{S}^1\mid \dim_H(p_e E) < s^\prime - 1/i\},
    \end{align*}
    and we have a contradiction by Theorem \ref{thm:exceptionalsetOptimalOracles}.
\end{proof}

\bibliographystyle{amsplain}
\bibliography{pdss}

\providecommand{\bysame}{\leavevmode\hbox to3em{\hrulefill}\thinspace}
\providecommand{\MR}{\relax\ifhmode\unskip\space\fi MR }
\providecommand{\MRhref}[2]{%
  \href{http://www.ams.org/mathscinet-getitem?mr=#1}{#2}
}
\providecommand{\href}[2]{#2}
\begin{thebibliography}{10}

\bibitem{AlbCsoPre10}
Giovanni Alberti, Marianna Cs{\"o}rnyei, and David Preiss, \emph{Differentiability of lipschitz functions, structure of null sets, and other problems}, Proceedings of the International Congress of Mathematicians 2010 (ICM 2010) (In 4 Volumes) Vol. I: Plenary Lectures and Ceremonies Vols. II--IV: Invited Lectures, World Scientific, 2010, pp.~1379--1394.

\bibitem{Bourgain10}
Jean Bourgain, \emph{The discretized sum-product and projection theorems}, J. Anal. Math. \textbf{112} (2010), 193--236. \MR{2763000}

\bibitem{CasLut15}
Adam Case and Jack~H. Lutz, \emph{Mutual dimension}, ACM Transactions on Computation Theory \textbf{7} (2015), no.~3, 12.

\bibitem{Chen18}
Changhao Chen, \emph{{Restricted Families of Projections and Random Subspaces}}, Real Analysis Exchange \textbf{43} (2018), no.~2, 347 -- 358.

\bibitem{FasOrp14}
Katrin F\"{a}ssler and Tuomas Orponen, \emph{On restricted families of projections in $\mathbb{R}^3$}, Proc. Lond. Math. Soc. (3) \textbf{109} (2014), no.~2, 353--381. \MR{3254928}

\bibitem{FieStu23}
Jacob~B. Fiedler and D.~M. Stull, \emph{Dimension of pinned distance sets for semi-regular sets}, 2023, arXiv:2309.11701.

\bibitem{FieStu24}
\bysame, \emph{Pinned distances of planar sets with low dimension}, 2024, arXiv:2408.00889.

\bibitem{GanGuoGuthHarMalWang2024}
Shengwen Gan, Shaoming Guo, Larry Guth, Terence L.~J. Harris, Dominique Maldague, and Hong Wang, \emph{On restricted projections to planes in $\mathbb{R}^3$}, 2024.

\bibitem{GanGuoWang2022}
Shengwen Gan, Shaoming Guo, and Hong Wang, \emph{A restricted projection problem for fractal sets in $\mathbb{R}^n$}, 2024.

\bibitem{GanGuthMal2023}
Shengwen Gan, Larry Guth, and Dominique Maldague, \emph{An exceptional set estimate for restricted projections to lines in $\mathbb r^3$}, The Journal of Geometric Analysis \textbf{34} (2023).

\bibitem{Hochman14}
Michael Hochman, \emph{{On self-similar sets with overlaps and inverse theorems for entropy}}, Annals of Mathematics \textbf{219} (2014), no.~2, 773--822.

\bibitem{KaeOrpVen17}
Antti K{\"a}enm{\"a}ki, Tuomas Orponen, and Laura Venieri, \emph{A marstrand-type restricted projection theorem in $\mathbb{R}^3$}, arXiv preprint arXiv:1708.04859 (2017).

\bibitem{Kaufman68}
Robert Kaufman, \emph{On {H}ausdorff dimension of projections}, Mathematika \textbf{15} (1968), 153--155. \MR{248779}

\bibitem{LutLut18}
Jack~H. Lutz and Neil Lutz, \emph{Algorithmic information, plane {K}akeya sets, and conditional dimension}, ACM Trans. Comput. Theory \textbf{10} (2018), no.~2, 7:1--7:22.

\bibitem{LutStu20}
Neil Lutz and D.~M. Stull, \emph{Bounding the dimension of points on a line}, Information and Computation \textbf{275} (2020), 104601.

\bibitem{LutStu22}
\bysame, \emph{Dimension spectra of lines}, Computability \textbf{11} (2022), no.~2, 85--112.

\bibitem{LutStu24}
Neil Lutz and D.M. Stull, \emph{Projection theorems using effective dimension}, Information and Computation \textbf{297} (2024), 105137.

\bibitem{Marstrand54}
J.~M. Marstrand, \emph{Some fundamental geometrical properties of plane sets of fractional dimensions}, Proc. London Math. Soc. (3) \textbf{4} (1954), 257--302. \MR{0063439}

\bibitem{Mattila75}
Pertti Mattila, \emph{Hausdorff dimension, orthogonal projections and intersections with planes}, Ann. Acad. Sci. Fenn. Ser. AI Math \textbf{1} (1975), no.~2, 227--244.

\bibitem{Orponen16}
Tuomas Orponen, \emph{Projections of planar sets in well-separated directions}, Adv. Math. \textbf{297} (2016), 1--25.

\bibitem{Orponen20a}
\bysame, \emph{Combinatorial proofs of two theorems of {L}utz and {S}tull}, arXiv preprint arXiv:2002.01743 (2020).

\bibitem{Orponen24}
Tuomas Orponen, \emph{On the projections of ahlfors regular sets in the plane}, 2024.

\bibitem{OrpRen24}
Tuomas Orponen and Kevin Ren, \emph{On the projections of almost ahlfors regular sets}, 2024, arXiv:2411.04528.

\bibitem{PerShm2007}
Yuval Peres and Pablo Shmerkin, \emph{Resonance between cantor sets}, Ergodic Theory and Dynamical Systems \textbf{29} (2007).

\bibitem{PraYangZahl}
Malabika Pramanik, Tongou Yang, and Joshua Zahl, \emph{A furstenberg-type problem for circles, and a kaufman-type restricted projection theorem in $\mathbb{R}^3$}, 2024.

\bibitem{Preiss90}
David Preiss, \emph{Differentiability of lipschitz functions on banach spaces}, Journal of Functional Analysis \textbf{91} (1990), no.~2, 312--345.

\bibitem{PreSpe15}
David Preiss and Gareth Speight, \emph{Differentiability of lipschitz functions in lebesgue null sets}, Inventiones mathematicae \textbf{199} (2015), no.~2, 517--559.

\bibitem{Stull22a}
D.~M. Stull, \emph{{Optimal Oracles for Point-To-Set Principles}}, 39th International Symposium on Theoretical Aspects of Computer Science (STACS 2022), Leibniz International Proceedings in Informatics (LIPIcs), vol. 219, 2022, pp.~57:1--57:17.

\bibitem{Stull22c}
D.~M. Stull, \emph{Pinned distance sets using effective dimension}, 2022, arXiv:2207.12501.

\end{thebibliography}

 \end{document}